\documentclass[a4paper, oneside, 11pt]{article}
\usepackage[english]{babel}
\usepackage{latexsym,amsfonts,amsmath,enumerate,graphics,enumerate,amsthm}
\title{\textbf{Order continuity from a topological perspective} }
\author{T.\ Hauser, A.\ Kalauch}

\let\epsilon=\varepsilon

\theoremstyle{definition}
\newtheorem{definition}{Definition}[section]
\newtheorem{theorem}[definition]{Theorem}
\newtheorem{proposition}[definition]{Proposition}
\newtheorem{lemma}[definition]{Lemma}
\newtheorem{corollary}[definition]{Corollary}
\theoremstyle{remark}
\newtheorem{remark}[definition]{Remark}
\newtheorem{example}[definition]{Example}

\let\epsilon=\varepsilon
\let\phi=\varphi
\let\theta=\vartheta
\newcommand{\R}{\mathbb{R}}
\newcommand{\N}{\mathbb{N}}

\begin{document}
\maketitle

\begin{abstract}
We study three types of order convergence and related concepts of order continuous maps in partially ordered sets, partially ordered abelian groups and partially ordered vector spaces, respectively. 
An order topology is introduced such that  
in the latter two settings under mild conditions order continuity is a topological property. 
We present a generalisation of the Ogasawara theorem on the structure of the set of order continuous operators.
\end{abstract}
 
\section{Introduction}

In this paper we deal with three types of order convergence, introduce an appropriate topology and relate these concepts. Moreover, we study the according four types of order continuity of maps and obtain properties of the corresponding sets of order continuous maps.
We investigate these concepts 
in partially ordered sets, in partially ordered abelian groups as well as in partially ordered vector spaces, where we intend to give the results as general as possible. 

The first concept of order convergence which we will deal with ($o_1$-convergence) is motivated by the usual order convergence in vector lattices, see e.g.\  \cite[Chapter 1, Section 4]{Positiveoperators_old} or 
\cite[Definition 1.1]{Abra}.
 For bounded nets, a definition of $o_1$-convergence can also be found in \cite[Chapter 1; Definition 5.1.]{Peressini67}.
In partially ordered vector spaces, $o_1$-convergence is considered  e.g.\ in 
\cite{IaB} and  \cite[Definition 1.7.]{Imh}.
 
The second and the third concept of order convergence ($o_2$-convergence, $o_3$-convergence) are given in vector lattices in \cite[p.\ 288]{Abra} and  
\cite[Definition 1.2]{Abra}, respectively.  
After introducing these concepts in partially ordered sets, we will show that $o_3$-convergence coincides with the convergence given in \cite[Definition 1]{Wolk} in partially ordered sets, and with the convergence introduced in 
\cite[Definition II.6.3.]{Vulikh67} in lattices. Our definition is inspired by
\cite[Definition 1.8.]{Imh}, where the concept is considered in
partially ordered vector spaces.

Operators in vector lattices that are continuous with respect to $o_1$-convergence
are frequently studied, see e.g.\ \cite[Definition 4.1]{Positiveoperators_old}, 
\cite[Definition 1.3.8]{Meyer91}. 
Operators on vector lattices that preserve $o_2$-convergence or $o_3$-convergence are considered in
\cite{Abra}.
Our aim is to introduce a concept of topology in partially ordered sets such that $o_1$-, $o_2$- and $o_3$-continuity, respectively, coincide with the topological continuity under mild conditions. 
Therefore we introduce an order topology $\tau_o$, which generalises the concept of order topology in partially ordered vector spaces given in \cite{Imh}. Note that 
$\tau_o$ is a special case of a $\sigma$-compatible topology on partially ordered sets considered in \cite{Floyd1955}.
We will show that $\tau_o$ coincides with the topology defined in lattices in \cite[Definition II.7.1]{Vulikh67} as well as  in \cite{Dobbertin84}. 

Note that another concept of topology, the so-called order bound topology, is introduced in partially ordered vector spaces in \cite[p.\ 20]{Namioka57}, see also \cite[Def.\ 2.66]{CAD}. In \cite[Theorem 5.2]{Namioka57} it is shown that each regular operator between partially ordered vector spaces is continuous with respect to the order bound topology. As there clearly exist examples of regular operators that are not $o_1$-continuous, the concept of order bound topology is not suitable for our purpose. 

The results in this paper are organised as follows. 
In Section 2 we introduce and characterise net catching sets and define $\tau_o$ in partially ordered sets. 
The three concepts of order convergence are defined in Section 3 in partially ordered sets. We link the concepts to the ones in the literature, show that the three concepts differ, investigate their relations and show that they imply $\tau_o$-convergence.
We prove that closedness with respect to $\tau_o$ is characterised by means of order convergence. Further properties of order convergence concepts such as monotonicity and a Sandwich theorem will be established.

In Section 4 we investigate maps that are continuous with respect to the order convergences and $\tau_o$-convergence, respectively, and relate these concepts. We show that  $o_3$-convergence in a lattice can be characterised by $o_2$-convergence in a Dedekind complete cover. 

In Section 5 we characterise the concepts of order convergence and net catching sets in partially ordered abelian groups.
Section 6 contains the Riesz-Kantorivich theorem in the setting of partially ordered abelian groups.

In Section 7, we give sufficient conditions on the domain and the codomain of an order bounded map between partially ordered abelian groups that guarantee the equivalence of
the four concepts of continuity. Under the same conditions, we show a generalisation of  
Ogasawara's theorem that can be found in \cite[Theorem 4.4]{Positiveoperators_old}, i.e.\ we prove that the set of all order bounded additive continuous maps is an order closed ideal in the lattice-ordered abelian group of all order bounded additive maps.  

In Section 8, we show that the scalar multiplication in partially ordered vector spaces is linked appropriately to the $o_i$-convergences if and only if the space is Archimedean and directed. Examples are given which show that the order convergences differ in this setting.
In Section 9 we show that the results of Section 7 are also valid for linear operators on partially ordered vector spaces.

Next we fix our notation.
As usual, on a non-empty set $P$ a binary relation $\leq$ is called a \emph{partial order} if it is reflexive, transitive and anti-symmetric. The set $P$ is then  called a \emph{partially ordered set}. 
For $x,y\in P$ we write $x<y$ if $x\leq y$ and $x\neq y$.
For $U,V\subseteq P$ we denote $U\leq V$ if for every $u\in U$ and $v\in V$ we have $u\leq v$. If $V=\{v\}$ for $v\in P$, we abbreviate $U\leq \{v\}$ by $U\leq v$ (and similarly $v\leq U$).
For $x \in P$ and $M \subseteq P$ define $M_{\geq x}:=\{m \in M;\, m \geq x\}$ and $M_{\leq x}:=\{m \in M;\, m \leq x\}$.
A set $M\subseteq P$ is called \emph{majorising} in $P$ if for every $x\in P$ the set $M_{\geq x}$ is non-empty.

For $x,y\in P$ the \emph{order interval} is given by $[x,y]:=\{z\in P; \, x\leq z\leq y\}$. 
$P$ is called \emph{directed (upward)} if for every $x,y\in P$ the set $P_{\geq x}\cap P_{\geq y}$ is non-empty. \emph{Directed downward} is defined analogously. A set $M\subseteq P$ is called \emph{full} if for every $x,y\in M $ one has $[x,y]\subseteq M$.
For a subset of $P$,  the notions \emph{bounded above}, \emph{bounded below},
\emph{order bounded},
  \emph{upper (or lower) bound} and  \emph{infimum (or supremum)} 
are defined as usual. For a net $(x_\alpha)_{\alpha \in A}$ in $P$ we denote $x_\alpha \downarrow$ if $x_\alpha \leq x_\beta$ whenever $\alpha\geq \beta$. For $x\in P$ we write $x_\alpha \downarrow x$ if $x_\alpha \downarrow$ and $\inf \{x_\alpha; \, \alpha \in A\}=x$. Similarly we define $x_\alpha \uparrow$ and $x_\alpha \uparrow x$.

$P$ is said to have the \emph{Riesz interpolation property} if for every non-empty 
finite sets $U,V\subseteq P$ with $U\leq V$ there is $x\in P$ such that $U\leq x\leq V$.
We call $P$ a \emph{lattice} if for every non-empty finite subset of $P$ the infimum and the supremum exist in $P$. A lattice $P$ is called \emph{Dedekind complete} if  
every non-empty set which is bounded above has a supremum, and every non-empty set which is bounded below has an infimum. 
We say that a lattice $P$ satisfies the \emph{infinite distributive laws} if 
for every $x\in P$ and $M\subseteq P$ the following equations hold
\begin{eqnarray}
x\wedge \left(\bigvee M\right)&=&\bigvee (x\wedge M),\nonumber\\
x\vee \left(\bigwedge M\right)&=&\bigwedge (x\vee M)\nonumber
\end{eqnarray}
(where in the first equation it is meant that if the supremum of the left-hand side of the equation exists, then also the one on the right-hand side, and both are equal).
If $P$ is a lattice which  satisfies the infinite distributive laws, then for $M,N\subseteq P$ the  formulas	
	\begin{eqnarray}\label{equ:distr_law}\left(\bigvee M\right)\wedge \left(\bigvee N\right)&=&
	\bigvee(M\wedge N)\\
\left(\bigwedge M\right)\vee \left(\bigwedge N\right)&=&
\bigwedge(M\vee N)\nonumber
	\end{eqnarray}
	are satisfied, see \cite[Chapter II.4]{Vulikh67}. 

The following statement is straightforward.
\begin{lemma} \label{lem:majosetsandsuppe}
	Let $P$ be a partially ordered set and $A \subseteq B \subseteq P$ such that $A$ is majorising in $B$. If the supremum of $B$ exists, then the supremum of $A$ exists and satisfies $\sup A = \sup B$.
\end{lemma}

We call $M\subseteq P$
\emph{order dense} in $P$ if for every $x\in P$ one has \[\sup M_{\leq x}=x=\inf M_{\geq x}.\] 
Clearly, every order dense subset of $P$ is majorising.
The next statement is shown for partially ordered vector spaces in \cite[Stelling 1.2.7]{Waaij}, for sake of completeness we give a shorter proof here.
\begin{proposition} \label{pro:orderdensitytransitive}
	Let $M \subseteq N\subseteq P$. If $M$ is order dense in $N$ and  $N$ is order dense in $P$, then $M$ is order dense in $P$.
\end{proposition}

\begin{proof}
	Let $p \in P$. Clearly, $p$ is a lower bound of $M_{\geq p}$. To show that $p$ is the greatest lower bound of $M_{\geq p}$, let $z\in P$ be another lower bound of $M_{\geq p}$. To obtain $p\geq z$, it is sufficient to show that $N_{\geq p}\subseteq N_{\geq z}$, since then the order density of $N$ in $P$ implies $p=\inf N_{\geq p}\geq \inf N_{\geq z}=z$. Let $n\in N_{\geq p}$. Then $M_{\geq n}\subseteq M_{\geq p}$, hence $z$ is a lower bound of  $M_{\geq n}$. As $M$ is order dense in $N$, we obtain $n=\inf M_{\geq n}\geq z$. Therefore $N_{\geq p}\subseteq N_{\geq z}$. We have shown $p=\inf M_{\geq p}$. A similar argument gives $p=\sup M_{\leq p}$.
\end{proof}

Let $P$ and $Q$ be partially ordered sets and $f\colon P\to Q$ a map. $f$ is called \emph{monotone} if for every $x,y\in P$ with $x\leq y$ one has that $f(x)\leq f(y)$, and 
\emph{order reflecting} if for every $x,y\in P$ with
 $f(x)\leq f(y)$ one has that $x\leq y$.
Note that every order reflecting map is injective.
We call $f$ an \emph{order embedding} if $f$ is monotone and order reflecting.
$f$ is called \emph{order bounded} if every order bounded set is mapped into an order bounded set.

In the next statement, for sets $U\subseteq P$ and $V\subseteq Q$ we use the notation $f[U]$ for the image of $U$ under $f$, and $[V]f$ for the preimage of $V$.

\begin{proposition} \label{pro:infimum}
Let $f\colon P\to Q$ be an order embedding and $M\subseteq P$.	
	\begin{itemize}
	\item[(i)]  If 
	the infimum of $f[M]$ exists in $Q$ and is an element of $f[P]$, then
	the infimum of $M$ exists in $P$ and equals the 
	unique preimage of $\inf f[M]$,
	i.e.\ $[\{\inf f[M]\}]f=\{\inf M\}$.
	\item[(ii)]  Assume that $f[P]$ is order dense in $Q$.
	Then the infimum of $M$ exists in $P$ if and only if
	the infimum of $f[M]$ exists in $Q$ and is an element of $f[P]$.
	\end{itemize}
	Analogous statements are valid for the supremum. 
	\end{proposition}
\begin{proof}
	For (i), assume that the infimum of $f[M]$ exists in $Q$ and
	is an element of $f[P]$. Since $f$ is injective, 
	there is a unique $p\in P$ with $f(p)=\inf f[M]$. 
	It is sufficient to show that
	$p=\inf M$. As $f$ is order reflecting, $p$ is a lower bound of $M$.
	For any other lower bound $l\in P$ of $M$ the monotony of $f$
	implies $f(l)$ to be a lower bound of $f[M]$. Thus 
	$f(l)\leq \inf f[M]=f(p)$. Since $f$ 
	is order reflecting, we conclude
	$l\leq p$. This proves $p$ to be the greatest lower bound of $M$, i.e.\
	$p = \inf M$. 
	
	In order to prove (ii), assume that the infimum of $M$ exists in $P$.
	We show that $f(\inf M)$ is the infimum of $f[M]$.
	The monotony of $f$ implies $f(\inf M)$ to be a lower bound of 
	$f[M]$. 
	Let $l\in Q$ be a lower bound of $f[M]$. 
	Since $f[P]$ is order dense 
	in $Q$, we know that $l=\sup \{q \in f[P]; q \leq l\}$. In order to prove 
	$l\leq f(\inf M)$ it is sufficient to show that $f(\inf M)$ is an upper bound of $\{q \in f[P]; q \leq l\}$.
	 For $q \in f[P]$ there is $p \in P$ such that
	$f(p)=q$. If furthermore $q \leq l$, we conclude $f(p)=q\leq l \leq f[M]$. 
	Since $f$ is order reflecting, $p$ is a lower bound of $M$.
	This implies $p \leq \inf M$, and the monotony of $f$ shows $q=f(p) \leq f(\inf M)$.
		We have therefore proven $f(\inf M)$ to be an upper bound of $\{q \in f[P]; q \leq l\}$. This implies $f(\inf M)$ to be the infimum of $f[M]$.	
		
		The statements about the supremum are shown analogously.
\end{proof} 	

Let $G$ be a partially ordered abelian group, 
i.e.\ $(G,+,0)$ is an abelian group with a partial order such that for every $x,y,z\in G$ with $x\le y$ it follows $x+z\leq y+z$.
Note that $G_+:=G_{\ge 0}$ is a monoid (with the induced operation from $G$). We call the elements of $G_+$ \emph{positive}.
$G_+$ is called \emph{generating}\footnote{As usual, for $M,N\subseteq G$ we define $M-N:=\{m-n; \, (m,n)\in M \times N\}.$} if $G=G_+-G_+$. Note that $G$ is directed if and only if $G_+$ is generating.  
We say that $G$ is \emph{Archimedean} if for every $x,y\in G$ with $nx\le y$
for all $n\in \N$ one has that $x\le 0$. 
 A directed full subgroup $I$ of $G$ is called an \emph{ideal}.
A subgroup $H$ of $G$ is full if and only if $H\cap G_+$ is full.

$G$ has the \emph{Riesz decomposition property} if for every $x,y\in G_+$ and $w\in [0,x+y]$ there are $u\in [0,x]$ and $v\in [0,y]$ such that $w=u+v$. 
  Observe that $G$ has the Riesz decomposition property if and only if $G$ has the Riesz interpolation property, see e.g.\ \cite[Proposition 2.1]{Goodearl1986}. 
 If $G$ is a lattice, then $G$ is called a \emph{lattice-ordered abelian group}. 
 	Note that every lattice-ordered abelian group satisfies the infinite distributive laws, see
 	\cite[Proposition 1.7]{Goodearl1986}, and hence the equations \eqref{equ:distr_law}. 
 	 For further standard notions in partially ordered abelian groups see \cite{Goodearl1986}.

Let $G$, $H$ be partially ordered abelian groups. 
We call a group homomorphism  $f\colon G\to H$ \emph{additive} and denote
the set of all additive maps  
from $G$ to $H$ by $\operatorname{A}(G,H)$.
As usual, on $\operatorname{A}(G,H)$ a group structure is introduced by means of $f+g\colon G\to H$, $x\mapsto f(x)+g(x)$, where the neutral element is  $0\colon x\mapsto 0$. A translation invariant pre-order on $\operatorname{A}(G,H)$ is defined by $f\leq g$ whenever for every $x\in G_+$ we have $f(x)\leq g(x)$. 
If $G$ is directed, then $\leq$ is a partial order on $\operatorname{A}(G,H)$.
Note that an element in  $\operatorname{A}(G,H)$ is positive if and only if it is monotone. 
We denote the set of all monotone maps in $\operatorname{A}(G,H)$ by $\operatorname{A}_{+}(G,H)$.
An element of the set $\operatorname{A}_{\operatorname{r}}(G,H):=\operatorname{A}_{+}(G,H)-\operatorname{A}_{+}(G,H)$ is called a \emph{regular} map. Finally, we denote
 the set of all order bounded maps in $\operatorname{A}(G,H)$ by $\operatorname{A}_{\operatorname{b}}(G,H)$.
Clearly, $\operatorname{A}_{\operatorname{r}}(G,H)\subseteq \operatorname{A}_{\operatorname{b}}(G,H)$. If $G$ is directed, then $\operatorname{A}(G,H)$, $\operatorname{A}_{\operatorname{b}}(G,H)$ and $\operatorname{A}_{\operatorname{r}}(G,H)$ are partially ordered abelian groups. 

On a real vector space $X$, we consider a partial order $\leq$ on $X$ such that $X$ is a partially ordered abelian group under addition, and for every $\lambda \in \mathbb{R}_+$ and $x\in X_+$ one has that $\lambda x\in X_+$. Then $X$ is called a \emph{partially ordered vector space}. 
 Note that $X$ is Archimedean if and only if $\frac{1}{n}x\downarrow 0$ for every $x \in X_+$.
If a partially ordered vector space $X$ is a lattice, we call $X$ a \emph{vector lattice}. For standard notations in the case that $X$ is a vector
lattice see \cite{Positiveoperators_old}.

If $X$ is an Archimedean directed partially ordered vector space, then there is an essentially unique Dedekind complete vector lattice $X^\delta$ and a linear order embedding $J\colon X \to X^\delta$ such that $J[X]$ is order dense in $X^\delta$. As usual, $X^\delta$ is called the \emph{Dedekind completion} of $X$.

For partially ordered vector spaces $X$ and $Y$, $\operatorname{L}(X,Y)$ denotes the space of all linear operators. 
We set $\operatorname{L}_+(X,Y)=\operatorname{A}_+(X,Y)\cap \operatorname{L}(X,Y)$, $\operatorname{L}_{\operatorname{r}}(X,Y)=\operatorname{A}_{\operatorname{r}}(X,Y)\cap \operatorname{L}(X,Y)$
and $\operatorname{L}_{\operatorname{b}}(X,Y)=\operatorname{A}_{\operatorname{b}}(X,Y)\cap \operatorname{L}(X,Y)$.
If $X$ is directed, $\operatorname{L}(X,Y)$, $\operatorname{L}_{\operatorname{b}}(X,Y)$ and $\operatorname{L}_{\operatorname{r}}(X,Y)$ are partially ordered vector spaces.

\section{Order topology in partially ordered sets} 
In this section, let $P$ be a partially ordered set. We will introduce the order  topology $\tau_o$ on $P$ using net catching sets, which we define next. 

\begin{definition}
	A subset $U\subseteq P$ is called a 
	\emph{net catching set}
	for $x\in P$ if for all nets  $(\hat{x}_\alpha)_{\alpha\in A}$ and $(\check{x}_\alpha)_{\alpha\in A}$ in $P$
	with $\hat{x}_\alpha \uparrow x$ and $\check{x}_\alpha\downarrow x$ there is
	$\alpha \in A$ such that $[\hat{x}_\alpha,\check{x}_\alpha]\subseteq U$.	
\end{definition}

\begin{proposition}\label{pro:net_catching_sets_pos}
	Let $U\subseteq P$ and $x\in P$. The following statements are equivalent.
	\begin{itemize}
		\item[(i)] $U$ is a net catching set for 
		 $x$.
		\item[(ii)] For all nets $(\hat{x}_\alpha)_{\alpha\in A}$ and $(\check{x}_\beta)_{\beta\in B}$ in $P$
		with $\hat{x}_\alpha \uparrow x$ and $\check{x}_\beta\downarrow x$ there are
		$\alpha \in A$ and $\beta \in B$ such that $[\hat{x}_\alpha,\check{x}_\beta]\subseteq U$.
		\item[(iii)] For all subsets $\hat{M}\subseteq P$ being directed upward and
		$\check{M}\subseteq P$ being directed downward with $\sup \hat{M}=x=\inf \check{M}$
		there are $\hat{m}\in \hat{M}$ and $\check{m}\in \check{M}$ such that $[\hat{m},\check{m}]\subseteq U$.
	\end{itemize}
\end{proposition}
\begin{proof}
	It is clear that (ii)$\Rightarrow$(i). In order to show  (i)$\Rightarrow$(iii), let 
	$\hat{M}$ and $\check{M}$ be as in (iii). We endow $\check{M}$ with the reversed order and define $A:=\hat{M}\times \check{M}$ with the component-wise order on $A$. For $\alpha=(\hat{m},\check{m})\in A$ let 
	$\hat{x}_\alpha := \hat{m}$ and
	$\check{x}_\alpha:=\check{m}$. This  defines nets $(\hat{x}_\alpha)_{\alpha \in A}$
	and $(\check{x}_\alpha)_{\alpha \in A}$ with $\hat{x}_\alpha \uparrow x$
	and $\check{x}_\alpha \downarrow x$. Thus (i) shows the existence of $(\hat{m},\check{m})= \alpha\in A$ such that
	$[\hat{m},\check{m}]=
	[\hat{x}_\alpha,\check{x}_\alpha] \subseteq U$. It remains to show 
	(iii)$\Rightarrow$(ii). Let 
	$(\hat{x}_\alpha)_{\alpha\in A}$ and $(\check{x}_\beta)_{\beta\in B}$ be as in (ii). Define $\hat{M}:=\{\hat{x}_\alpha;\alpha \in A\}$ and $\check{M}:=\{\check{x}_\beta;\beta \in B\}$ and observe that $\hat{M}$
	is directed upward and $\check{M}$ is directed downward with $\sup \hat{M}=x=\inf \check{M}$. From (iii) we conclude the existence of $\hat{m}\in \hat{M}$ and $\check{m}\in \check{M}$ such that
	$[\hat{m},\check{m}]\subseteq U$. There are $\alpha \in A$ and $\beta \in B$ such that $\hat{m}=\hat{x}_\alpha$ and $\check{m}=\check{x}_\beta$, which implies
	$[\hat{x}_\alpha,\check{x}_\beta]=[\hat{m},\check{m}]\subseteq U$. 
\end{proof}
\begin{definition}
	A subset $O$ of $P$ is called \emph{order open} if $O$ is a net catching set for every $x\in O$. A subset $C$ of $P$ is called \emph{order closed} if $P\setminus C$ is order open. Define
	\begin{align*}
	\tau_o(P):=\{O\subseteq P;\, O \mbox{ is order open}\}.
	\end{align*}
\end{definition}
The following is straightforward.
\begin{proposition}
	$\tau_o(P)$ is a topology on $P$.
\end{proposition}	
The topology $\tau_o(P)$ (or, shortly, $\tau_o$) is referred to as the \emph{order topology} on $P$. As usual, for a net $(x_\alpha)$ in $P$ converging to $x\in P$ with respect to the topology $\tau_o$ we write
$x_\alpha \xrightarrow{\tau_o} x$.
\begin{remark}
Our definition of the order topology is a straightforward generalisation of the topology given in \cite{Dobbertin84} on complete lattices.
For this, compare \cite[Proposition 1]{Dobbertin84} with \ref{pro:net_catching_sets_pos} (iii).

On the other hand, note that a net catching set is a generalisation of a concept in partially ordered vector spaces introduced in \cite[Definition 3.3]{Imh}. 
By \cite[Theorem 4.2]{Imh}, the order topology coincides with the topology studied in \cite{Imh}.
\end{remark}

\section{Order convergence in partially ordered sets} 
In this section, let $P$ be a partially ordered set. We will introduce three types of order convergence and relate them to $\tau_o$-convergence. 

\begin{definition} \label{def:orderconvergences}
Let $x \in P$ and let
$(x_\alpha)_{\alpha \in A}$ be a net in $P$.
We define
\begin{itemize}
	\item[(i)]
	$x_\alpha \xrightarrow{o_1} x$, 
	if there are nets $(\hat{x}_\alpha)_{\alpha \in A}$
	and $(\check{x}_\alpha)_{\alpha \in A}$ in $P$ such that $\check{x}_\alpha \downarrow x$,
	$\hat{x}_\alpha \uparrow x$ and $\hat{x}_\alpha \leq x_\alpha \leq \check{x}_\alpha$ for every 
	$\alpha \in A$.
	\item[(ii)] 
	$x_\alpha \xrightarrow{o_2} x$, 
	if there are nets $(\hat{x}_\alpha)_{\alpha \in A}$
	and $(\check{x}_\alpha)_{\alpha \in A}$ in $P$ and $\alpha_0 \in A$ such that $\check{x}_\alpha \downarrow x$,
	$\hat{x}_\alpha \uparrow x$ and $\hat{x}_\alpha \leq x_\alpha \leq \check{x}_\alpha$ for every 
	$\alpha \in A_{\geq \alpha_0}$.
	\item[(iii)] 
	$x_\alpha \xrightarrow{o_3} x$, 
	if there are nets $(\hat{x}_\beta)_{\beta \in B}$
	and $(\check{x}_\gamma)_{\gamma \in C}$ in $P$ and a map 
	$\eta\colon B \times C \rightarrow A$ such that $\hat{x}_\beta \uparrow x$,
	$\check{x}_\gamma \downarrow x$ and $\hat{x}_\beta \leq x_\alpha \leq \check{x}_\gamma$ for every 
	$\beta\in B$, $\gamma \in C$ and $\alpha \in A_{\geq \eta(\beta,\gamma)}$.
\end{itemize}
\end{definition}

\begin{remark} \label{rem:linktoliteratureoiconv}
	Note that the $o_1$-convergence is inspired by the classical order convergence in vector lattices, see e.g.\ \cite{Positiveoperators_old}.
	The concepts of $o_2$-convergence and $o_3$-convergence are adopted from \cite{Abra}, where these convergences are considered in vector lattices. In Proposition \ref{pro:char_o_i_poag} below the precise link will be given. The $o_3$-convergence in partially ordered vector spaces is defined in \cite[Section 1.4]{Wulich2017}. Note furthermore that the order convergence concepts studied  in \cite[II.6.3]{Vulikh67} for lattices and in \cite[Definition 1]{Wolk} for partially ordered sets are equivalent to the $o_3$-convergence. This will be established in Proposition \ref{pro:char_o3conv} below.  
\end{remark}

To establish the link to the order convergence concepts given in \cite{Wulich2017} and \cite{Wolk}, we need the following notion.

\begin{definition}
	Let $M$ be a set. A net $(x_\alpha)_{\alpha \in A}$ is called a \emph{direction} if for arbitrary $\alpha \in A$ there is $\beta \in A$ such that $\alpha < \beta$.
\end{definition}

The next lemma gives a link between directions and nets. 

\begin{lemma} \label{lem:directions}
	Let $M$ be a set and let $(x_\alpha)_{\alpha \in A}$ be a net in $M$. If $A \times \mathbb{N}$ is ordered componentwise, $(x_\alpha)_{(\alpha,n)\in A\times \mathbb{N} }$ is a direction and a subnet of $(x_\alpha)_{\alpha \in A}$. 
\end{lemma}

\begin{proof} Clearly $(x_\alpha)_{(\alpha,n)\in A\times \mathbb{N} }$ is a direction. The map $\phi \colon A\times \mathbb{N} \to A$, $(\alpha,n)\mapsto \alpha$ is monotone and $\phi[A\times \mathbb{N}]$ is majorising in $A$. Since $x_\alpha=x_{\phi(\alpha,n)}$ for every $(\alpha,n)\in A \times \mathbb{N}$, the net $(x_\alpha)_{(\alpha,n)\in A\times \mathbb{N} }$ is a subnet of $(x_\alpha)_{\alpha \in A}$.
\end{proof}

In the subsequent proposition, the statement in (iii) is the convergence given in \cite[Definition II.6.3]{Vulikh67}, and the concept in (iv) is the convergence considered in \cite[Definition 1]{Wolk}.

\begin{proposition}\label{pro:char_o3conv}
	Let 
	$x \in P$ and let 
	$(x_\alpha)_{\alpha \in A}$ be a net in $P$. 
	Then the following statements are equivalent.
	\begin{itemize}
		\item[(i)] $x_\alpha \xrightarrow{o_3} x$,
		\item[(ii)] there are nets $(\hat{x}_\beta)_{\beta \in B}$
		and $(\check{x}_\beta)_{\beta \in B}$ in $P$ and a map 
		$\eta\colon B \rightarrow A$ such that $\hat{x}_\beta \uparrow x$,
		$\check{x}_\beta \downarrow x$ and $\hat{x}_\beta \leq x_\alpha \leq \check{x}_\beta$ for every 
		$\beta\in B$ and $\alpha \in A_{\geq \eta(\beta)}$,
		\item[(iii)] there are directions $(\hat{x}_\beta)_{\beta \in B}$
		and $(\check{x}_\gamma)_{\gamma \in C}$ in $P$ and a map 
		$\eta\colon B \times C \rightarrow A$ such that $\hat{x}_\beta \uparrow x$,
		$\check{x}_\gamma \downarrow x$ and $\hat{x}_\beta \leq x_\alpha \leq \check{x}_\gamma$ for every 
		$\beta\in B$, $\gamma \in C$ and $\alpha \in A_{\geq \eta(\beta,\gamma)}$. 
		\item[(iv)] there are sets $\hat{M},\check{M}\subseteq P$ and $\kappa \colon \hat{M}\times \check{M}\rightarrow A$
		such that $\hat{M}$ is directed upward, $\check{M}$ is directed downward, $\sup \hat{M}=x=\inf \check{M}$
		and for every $\hat{m} \in \hat{M}$, $\check{m}\in \check{M}$ and
		$\alpha \in  A_{\geq \kappa(\hat{m},\check{m})}$ we have $\hat{m}\leq x_\alpha \leq \check{m}$.
	\end{itemize}
	
	\begin{proof}
		It is clear that (ii) implies (i) and that (iii) implies (i).
		To show that (i) implies (ii), 
		we assume that there are nets $(\hat{x}_\beta)_{\beta \in B}$
		and $(\check{x}_\gamma)_{\gamma \in C}$ in $P$ and a map 
		$\eta\colon B \times C \rightarrow A$ such that $\hat{x}_\beta \uparrow x$,
		$\check{x}_\gamma \downarrow x$ and $\hat{x}_\beta \leq x_\alpha \leq \check{x}_\gamma$ for every 
		$\beta\in B$, $\gamma \in C$ and $\alpha \in A_{\geq \eta(\beta,\gamma)}$. For $(\beta,\gamma)\in B\times C$ we define $\hat{y}_{(\beta,\gamma)}:=\hat{x}_\beta$ and $\check{y}_{(\beta,\gamma)}:=\check{x}_\gamma$. Observe that $(\hat{y}_\delta)_{\delta\in B\times C}$ is a subnet of $(\hat{x}_\beta)_{\beta \in B}$ and, similarly, $(\check{y}_\delta)_{\delta\in B\times C}$ is a subnet of $(\check{x}_\gamma)_{\gamma \in C}$. Thus $\hat{y}_\delta\uparrow x$ and $\check{y}_\delta\downarrow x$. Furthermore, for $(\beta,\gamma)\in B\times C$ and $\alpha\in A_{\geq \eta(\beta,\gamma)}$ we have $\hat{y}_{(\beta,\gamma)}=\hat{x}_\beta\leq x_\alpha\leq \check{x}_\gamma=\check{y}_{(\beta,\gamma)}$.
		
		We next show that (i) implies (iii).  Let 
		$(\hat{x}_\beta)_{\beta\in B}$,  $(\check{x}_\gamma)_{\gamma \in C}$ and $\eta\colon B \times C \to A$ be as in Definition \ref{def:orderconvergences}. According to Lemma \ref{lem:directions} we consider the directions $(\hat{x}_\beta)_{(\beta,n)\in B\times \mathbb{N}}$,  $(\check{x}_\gamma)_{(\gamma,m) \in C\times \mathbb{N}}$ and define $\tilde{\eta}\colon(B\times \mathbb{N})\times (C\times \mathbb{N})$, $((\beta,n),(\gamma,m))\mapsto\eta(\beta,\gamma)$ to obtain (iii).
				
		To show that (i) implies (iv), set $\hat{M}:=\{\hat{x}_\beta; \beta\in B\}$ and $\check{M}:=\{\check{x}_\gamma; \gamma\in C\}$ and observe that $\hat{M}$ is directed upward, $\check{M}$ is directed downward and $\sup \hat{M}=x=\inf \check{M}$ is satisfied.		
		To construct $\kappa$, note that for $(\hat{m},\check{m})\in\hat{M}\times\check{M}$ there is $(\beta,\gamma)\in B\times C$ such that $\hat{m}=\hat{x}_\beta$ and
		$\check{m}=\check{x}_\gamma$. Hence we can define $\kappa(\hat{m},\check{m}):=\eta(\beta,\gamma)$ and obtain for
	$\alpha \in  A_{\geq \kappa(\hat{m},\check{m})}=A_{\geq \eta(\beta,\gamma)}$ that $\hat{m}=\hat{x}_\beta\leq x_\alpha \leq \check{x}_\gamma=\check{m}$.
		
		Finally we establish that (iv) implies (i). Define $B:=\hat{M}$, $C:=\check{M}$, where $C$ is endowed with the reversed order of $P$. For $\beta\in B$ and $\gamma\in C$ set $\hat{x}_\beta:=\beta$ and $\check{x}_\gamma:=\gamma$, moreover define $\eta:=\kappa$, which yield the desired properties.
	\end{proof}
\end{proposition}

The following proposition gives the general relationships between the different concepts of order convergence. The further discussion below will show that all the concepts differ. 

\begin{proposition}\label{pro:basic_convergences}
		Let 
		$x \in P$ and let 
		$(x_\alpha)_{\alpha \in A}$ be a net in $P$. Then
		\begin{itemize}
			\item[(i)] $x_\alpha \xrightarrow{o_1} x$ implies $x_\alpha \xrightarrow{o_2} x$,
			\item[(ii)] $x_\alpha \xrightarrow{o_2} x$ implies $x_\alpha \xrightarrow{o_3} x$, and
			\item[(iii)] $x_\alpha \xrightarrow{o_3} x$ implies $x_\alpha \xrightarrow{\tau_o} x$.
	\end{itemize}
\end{proposition}
\begin{proof}
	As (i) and (ii) are straightforward, it remains to show (iii). For this, let $O\in \tau_o$ be a  neighbourhod of $x$. The convergence $x_\alpha \xrightarrow{o_3} x$ means that
	there are nets $(\hat{x}_\beta)_{\beta \in B}$
	and $(\check{x}_\gamma)_{\gamma \in C}$ in $P$ and a map 
	$\eta\colon B \times C \rightarrow A$ such that $\hat{x}_\beta \uparrow x$,
	$\check{x}_\gamma \downarrow x$ and $\hat{x}_\beta \leq x_\alpha \leq \check{x}_\gamma$ for every 
	$\beta\in B$, $\gamma \in C$ and $\alpha \in A_{\geq \eta(\beta,\gamma)}$. Since $O$ is a net catching set for $x$, Proposition \ref{pro:net_catching_sets_pos} shows the existence of $\beta\in B$ and $\gamma\in C$ such that $[\hat{x}_\beta,\check{x}_\gamma]\subseteq O$. Hence for $\alpha\in A_{\geq \eta(\beta,\gamma)}$ we have $x_\alpha\in [\hat{x}_\beta,\check{x}_\gamma]\subseteq O$. 
\end{proof}

\begin{remark}\label{rem:decreasingnet}
	(a) Observe that every net $(x_\alpha)_{\alpha\in A}$ with $x_\alpha\downarrow x\in P$ satisfies $x_\alpha\xrightarrow{o_1}x$, and due to Proposition \ref{pro:basic_convergences} also $x_\alpha\xrightarrow{\tau_o}x$.\\
	(b) Let $M \subseteq P$, let   $(x_\alpha)_{\alpha \in A}$ be a net in $M$ and let $i \in \{1,2,3\}$. Note that if $x_\alpha \xrightarrow{o_i}x\in M$ in $M$, then also $x_\alpha \xrightarrow{o_i}x$ in $P$. An analogue is valid for $\tau_o$-convergence. Note furthermore that the converse statements are not true, in general. This is shown in Example \ref{exa:extensionprop} below, where $M$ is even an order dense subspace of a vector lattice $P$.
\end{remark}

\begin{remark}\label{rem:o_1_and_o_2}
	Let $(x_\alpha)_{\alpha\in A}$ be a net in $P$ and $x\in P$. We have $x_\alpha\xrightarrow{o_2}x$ if and only if there is $\alpha\in A$ such that the net $(x_\beta)_{\beta\in A_{\geq \alpha}}$ satisfies  $x_\beta\xrightarrow{o_1}x$. 
\end{remark}

In general, $o_2$-convergence does not imply $o_1$-convergence.

\begin{proposition} \label{pro:o2noto1}
	Let $x\in P$ have the property that for every $p\in P_{\geq x}$ there is a $q\in P$ such that $p<q$. Then  there is a net $(x_\alpha)_{\alpha\in A}$ in $P$ and  such that $x_\alpha\xrightarrow{o_2}x$, but not $x_\alpha\xrightarrow{o_1}x$.   
\end{proposition}
\begin{proof}
	Let $x\in P$ have the above property.
	Consider $A:=P_{\geq x}$ and define a partial order $\preceq$ on $A$, where  on $A\setminus \{x\}$  the induced order from $P$ is taken. Moreover, define for every $y\in A$ that $y\preceq x$. Observe that $A$ is directed upward. Set $x_\alpha:=\alpha$ for every $\alpha\in A$. First we show $x_\alpha\xrightarrow{o_2}x$. We define $\alpha_0:=x$ and $\hat{x}_{\alpha}:=\check{x}_{\alpha}:=x$ for every $\alpha\in A$ and obtain $\hat{x}_\alpha \leq x_\alpha \leq \check{x}_\alpha$ for every $\alpha\in A_{\succeq \alpha_0}=\{x\}$.
	
	It remains to show that $x_\alpha\xrightarrow{o_1}x$ does not hold. Assuming the contrary, there is a net $(\check{x}_{\alpha})_{\alpha\in A}$ with $\check{x}_\alpha\downarrow x$ 
	and $x_\alpha\leq\check{x}_\alpha$ for every $\alpha\in A$.
	By the assumption, there is $\alpha \in A$ such that $\alpha>x$ and $\beta\in A$ such that $\beta> \check{x}_{\alpha}\in A$. Observe that $\beta\geq \check{x}_\alpha\geq x_\alpha=\alpha>x$, hence $\beta\succeq \alpha$ and thus $\beta>\check{x}_\alpha\geq \check{x}_\beta\geq x_\beta=\beta$, which is a contradiction.    
\end{proof}

\begin{remark}\label{rem:o1o2}
	(a) Assume $P$ to be directed upward and downward, $(x_\alpha)_{\alpha \in A}$ to be a net in $P$ such that $\{x_\alpha;\, \alpha \in A\}$ is bounded, and $p \in P$. Then $x_\alpha \xrightarrow{o_1}p$ if and only if $x_\alpha \xrightarrow{o_2}p$. 
	
	One implication follows from Proposition \ref{pro:basic_convergences}. To show the other one, let $x_\alpha \xrightarrow{o_2}p$. Thus there are nets $(\hat{x}_\alpha)_{\alpha \in A}$ and $(\check{x}_\alpha)_{\alpha\in A}$ and $\alpha_0 \in A$ such that $\hat{x}_\alpha \uparrow p$, $\check{x}_\alpha \downarrow p$ and $\hat{x}_\alpha \leq x_\alpha \leq \check{x}_\alpha$ for all $\alpha \in A_{\geq \alpha_0}$. Since $P$ is directed upward and $\{x_\alpha;\, \alpha \in A\}$ is bounded, there is an upper bound $\check{p}$ of $\{x_\alpha ;\, \alpha \in A\}\cup \{\check{x}_{\alpha_0}\}$.
	For $\alpha \in A$ define $\check{y}_\alpha:=\check{x}_\alpha$ if $\alpha \geq \alpha_0$ and $\check{y}_\alpha:=\check{p}$ otherwise. This defines a net $(\check{y}_\alpha)_{\alpha \in A}$ with $\check{y}_\alpha \downarrow p$ and $x_\alpha \leq \check{y}_\alpha$ for every $\alpha \in A$. Similarly we can define a net $(\hat{y}_\alpha)_{\alpha \in A}$ to obtain $x_\alpha \xrightarrow{o_1}p$.
	
	(b) The statement in (a) shows that the definition of order convergence given in \cite[Chapter 1, Section 5]{Peressini67} for nets with bounded domain coincides with the concepts of $o_1$-convergence and $o_2$-convergence.
	
	(c) If $x\in P$ is such that $P_{\geq x}$  is directed upward and $P_{\leq x}$  is directed downward, then the following are equivalent:
	\begin{itemize}
		\item[(i)] For every net $(x_\alpha)_{\alpha\in A}$ in $P$ with $x_\alpha\xrightarrow{o_2} x$ we have that $x_\alpha\xrightarrow{o_1} x$.
		\item[(ii)] $P_{\geq x}$ is bounded above and $P_{\leq x}$ is bounded below.
	\end{itemize}
Indeed, to show (i)$\Rightarrow$(ii), we assume, to the contrary, that (ii) is not valid. 
Suppose w.l.o.g.\ that $P_{\geq x}$ is not bounded from above, thus for every $p\in P_{\geq x}$ there is $r\in P_{\geq x}$ such that $r\not\leq p$. Since $P_{\geq x}$ is directed upward, there is $q\in P_{\geq x}$ such that $p,r\leq q$. As $p<q$, the assumption of Proposition \ref{pro:o2noto1} is satisfied, i.e.\ (i) is not true.
  
We establish (ii)$\Rightarrow$(i). Let  $(x_\alpha)_{\alpha\in A}$ be a net in $P$ such that  $x_\alpha\xrightarrow{o_2} x$, i.e.\ there are nets $(\hat{x}_\alpha)_{\alpha \in A}$
and $(\check{x}_\alpha)_{\alpha \in A}$ in $P$ and $\alpha_0 \in A$ such that $\check{x}_\alpha \downarrow x$,
$\hat{x}_\alpha \uparrow x$ and $\hat{x}_\alpha \leq x_\alpha \leq \check{x}_\alpha$ for every 
$\alpha \in A_{\geq \alpha_0}$. By (ii) there is an upper bound $u\in P$ for $P_{\geq x}$ and a lower bound $l\in P$ for $P_{\leq x}$. For $\alpha\in A$, set $\check{y}_\alpha:=\check{x}_\alpha$ whenever $\alpha\geq\alpha_0$, and $\check{y}_\alpha:=u$ otherwise. Similarly, set $\hat{y}_\alpha:=\hat{x}_\alpha$ whenever $\alpha\geq\alpha_0$, and $\hat{y}_\alpha:=l$ otherwise. Observe that $\check{y}_\alpha \downarrow x$,
$\hat{y}_\alpha \uparrow x$ and $\hat{y}_\alpha \leq x_\alpha \leq \check{y}_\alpha$ for every 
$\alpha \in A$. Thus $x_\alpha\xrightarrow{o_1} x$. 	
\end{remark}


\begin{remark} 
	Due to Remark \ref{rem:o1o2}(c), in every  partially ordered vector space the concepts of $o_1$-convergence and $o_2$-convergence differ.  
	Furthermore, an example of Fremlin in \cite[Example 1.4]{Abra} shows that
	$o_3$-convergence does not imply $o_2$-convergence. For this, use Proposition \ref{pro:char_o_i_poag}	below.  A sequence which is $\tau_o$-convergent, but not $o_3$-convergent, can be found in Example \ref{exa:ordertopconvergentnetnoto3convergent} below. The last two  examples are given in the setting of vector lattices. Note that there are examples where $o_2$-convergence, $o_3$-convergence and $\tau_o$-convergence coincide, see Example \ref{exa:opensubsetsofR} below.
\end{remark}

\begin{proposition}
	\label{pro:o2_o3_convergenceDedekind_complete_lattice}
	Let $P$ be a Dedekind complete lattice, let $(x_\alpha)_{\alpha \in A}$ be a net in $P$ and $x \in P$. Then $x_\alpha \xrightarrow{o_2}x$ if and only if $x_\alpha \xrightarrow{o_3}x$.
\end{proposition}

\begin{proof}
	Due to Proposition \ref{pro:basic_convergences} it is sufficient to show that $x_\alpha \xrightarrow{o_3}x$ implies $x_\alpha \xrightarrow{o_2}x$. Assume that there are nets $(\hat{x}_\beta)_{\beta \in B}$
	and $(\check{x}_\gamma)_{\gamma \in C}$ in $P$ and a map 
	$\eta\colon B \times C \rightarrow A$ such that $\hat{x}_\beta \uparrow x$,
	$\check{x}_\gamma \downarrow x$ and $\hat{x}_\beta \leq x_\alpha \leq \check{x}_\gamma$ for every 
	$\beta\in B$, $\gamma \in C$ and $\alpha \in A_{\geq \eta(\beta,\gamma)}$. Fix $(\beta_0,\gamma_0)\in B \times C$. 
	Set $\alpha_0:=\eta(\beta_0,\gamma_0)$. By Remark \ref{rem:o_1_and_o_2} it is sufficient to prove that $(x_\alpha)_{\alpha \in A_{\geq \alpha_0}}$
	is $o_1$-convergent to $x$.
	
	For $\alpha \in A$ define 
	$M_\alpha:=\{x_{\kappa}; \, \kappa \in A_{\geq \alpha}\}\cup \{x\}$. 
	Note that for $(\beta,\gamma) \in B \times C$ and $\alpha \in A_{\geq \eta(\beta,\gamma)}$ we have 
	$\hat{x}_\beta \leq M_\alpha\leq \check{x}_{\gamma}$. As $P$ is a Dedekind complete lattice, $\hat{y}_\alpha:=\inf M_\alpha$
	and $\check{y}_\alpha:=\sup M_\alpha$ exist
	for $\alpha \in A_{\geq \alpha_0}$. Furthermore $\hat{y}_\alpha \leq \{x_\alpha,x\} \leq \check{y}_\alpha$ 
	for all $\alpha \in A_{\geq \alpha_0}$, $\hat{y}_\alpha \uparrow$ and $\check{y}_\alpha \downarrow$. 
	Let $\hat{y}_\alpha\leq z$ for all $\alpha \in A_{\geq \alpha_0}$. For $\beta \in B$ there is $\alpha \in A_{\geq \alpha_0}$ such that 
	$\eta(\beta,\gamma_0)\leq \alpha$. Hence $\hat{x}_\beta \leq \inf M_{\eta(\beta,\gamma_0)}\leq \inf M_\alpha= \hat{y}_\alpha\leq z$ and we obtain
	$x=\sup\{\hat{x}_\beta;\, \beta \in B\}\leq z$. This shows $\hat{y}_\alpha \uparrow x$. Analogously we get $\check{y}_\alpha \downarrow x$. 
\end{proof}

If we introduce the order topology $\tau_o$ on the partially ordered set of real numbers $\mathbb{R}$, we obtain the standard topology on $\mathbb{R}$. 

\begin{example} \label{exa:opensubsetsofR}
	Let $M \subseteq \mathbb{R}$ be an open set with respect to the standard topology $\tau$ and equip $M$ with the standard order of $\mathbb{R}$. We show that $\tau_o(M)$ is the restriction $\tau(M)$ of $\tau$ to $M$ and that $o_2$- and $o_3$-convergence in $M$ coincide with the  convergence with respect to $\tau(M)$. Note that from Remark \ref{rem:o1o2} (c) 
	it follows that $o_1$-convergence and $o_2$-convergence in $M$ do not coincide. We first show that convergence with respect to $\tau(M)$ implies $o_2$-convergence. Indeed, let $(x_\alpha)_{\alpha\in A}$ be a net in $M$ such that $x_\alpha \xrightarrow{\tau(M)}x \in M$. Since $M$ is open, there is $r>0$ such that  the open ball $B_r(x)\subseteq \mathbb{R}$ with center $x$  and radius $r$ is contained in $M$. Hence there is $\alpha_0\in A$ such that for every $\alpha \in A_{\geq \alpha_0}$ we have $x_\alpha \in B_r(x)$. We therefore assume w.l.o.g.\ that $(x_\alpha)_{\alpha\in A}$ is a net in $B_r(x)$. Since $B_r(x)$ is a Dedekind complete lattice, by Proposition \ref{pro:o2_o3_convergenceDedekind_complete_lattice}
	it is sufficient to show that $x_\alpha \xrightarrow{o_3}x$. 
	For $\beta \in B:=(0,r)$
	let  $\hat{x}_\beta:= x-\beta$ and $\check{x}_\beta:=x+ \beta$. If we equip $B$ with the reversed order of $\mathbb{R}$, we obtain nets $(\hat{x}_\beta)_{\beta \in B}$ and $(\check{x}_\beta)_{\beta \in B}$ in $B_r(x)$ with $\hat{x}_\beta \uparrow x$ and $\check{x}_\beta \downarrow x$. For every $\beta \in B$ there is $\alpha_\beta \in A$ such that for every $\alpha \in A_{\geq \alpha_\beta}$ we have $|x_\alpha-x|\leq \beta$, i.e.\ $\hat{x}_\beta\leq x_\alpha \leq \check{x}_\beta$. We set $\eta \colon B \to A$, $\beta \mapsto \alpha_\beta$, and obtain  $x_\alpha \xrightarrow{o_3}x$. We have now shown that convergence with respect to $\tau(M)$ implies $o_2$-convergence in $M$. Note that $o_2$-convergence implies $o_3$-convergence and that $o_3$-convergence implies convergence with respect to $\tau_o(M)$ in $M$ by Proposition \ref{pro:basic_convergences}. It therefore remains to establish that convergence with respect to $\tau_o(M)$ implies convergence with respect to $\tau(M)$. To show that $\tau(M) \subseteq \tau_o(M)$, let $O \in \tau(M)$ and $x \in O$.  Since $M$ is open in $\R$ with respect to $\tau$ and $O \in \tau(M)$, we conclude $O \in \tau$. Thus there is $r>0$ such that $B_{2r}(x)\subseteq O$. To show that $O$ is a net catching set for $x$ let  $(\hat{x}_\alpha)_{\alpha \in A}$ and $(\check{x}_\alpha)_{\alpha \in A}$ be nets in $M$ such that $\hat{x}_\alpha \uparrow x$ and $\check{x}_\alpha \downarrow x$. Thus there is $\alpha \in A$ such that $[\hat{x}_\alpha,\check{x}_\alpha]\subseteq [x-r,x+r]\subseteq B_{2r}(x)\subseteq O$. This proves $O \in \tau_o(M)$.
\end{example}

Order closed sets can be characterised by means of $o_i$-convergence.

\begin{theorem} 
	\label{thm:orderclosed}
	Let $i\in \{1,2,3\}$ and  $C\subseteq P$. The following statements are equivalent:
	\begin{itemize}
		\item[(i)] $C$ is order closed.
		\item[(ii)] For every net $(x_{\alpha})_{\alpha\in A}$ in $C$ with $x_\alpha\xrightarrow{o_i} x\in P$ it follows that $x\in C$.
	\end{itemize} 
\end{theorem}
\begin{proof}
	In this proof, a set $C$ that satisfies (ii) is called $o_i$-closed.
	Observe that from 
	Proposition \ref{pro:basic_convergences} it follows that order closed sets are always $o_3$-closed, $o_3$-closed sets are $o_2$-closed and that $o_2 $-closed sets are $o_1$-closed. It
	remains to show that $o_1$-closed sets are order closed. By contradiction, assume that  $C\subseteq P$ is not order closed. Thus $P\setminus C$ is not order open, i.e.\ there is $x\in P\setminus C$ such that $P\setminus C$ is not a net catching set for $x$. This implies the existence of 	
	nets  $(\hat{x}_\alpha)_{\alpha\in A}$ and $(\check{x}_\alpha)_{\alpha\in A}$ in $P$
	with $\hat{x}_\alpha \uparrow x$ and $\check{x}_\alpha\downarrow x$ such that for every 
	$\alpha \in A$ we have that $[\hat{x}_\alpha,\check{x}_\alpha]\not\subseteq P\setminus C$. Hence, for every $\alpha\in A$ there is $x_\alpha\in[\hat{x}_\alpha,\check{x}_\alpha]\cap C$. Note that $(x_\alpha)_{\alpha\in A}$ is a net in $C$ with $x_\alpha\xrightarrow{o_1}x\in P\setminus C$, hence $C$ is not $o_1$-closed.	
\end{proof}

\begin{corollary}
	\label{cor:orderdenseimpliesordertopologicdense}
	Let $M\subseteq P$ be a lattice with the induced order from $P$. If $M$ is order dense in $P$, then $M$ is dense in $P$ with respect to $\tau_o(P)$.
\end{corollary}

\begin{proof}
	Let $p \in P$. Let $A:=M_{\geq p}$ be equipped with the reversed order of $M$. 
	Since $M$ is a lattice, we know $A$ to be directed. Setting $x_\alpha:=\alpha$ for $\alpha \in A$, we obtain a net $(x_\alpha)_{\alpha \in A}$ in $M$ with $x_\alpha \downarrow$. Since $M$ is order dense in $P$, we know furthermore $\inf\{x_\alpha;\, \alpha \in A\}=\inf A=\inf M_{\geq p}=p$, hence $x_\alpha \downarrow p$.
	Thus $x_\alpha \xrightarrow{o_1}p$ and Theorem \ref{thm:orderclosed} shows that $p$ is contained in the closure of $M$ with respect to $\tau_o(P)$. 
\end{proof}


For $o_i$-limits, we obtain the following monotonicity property.

\begin{proposition}\label{pro:monotony}
	Let $i\in\{1,2,3\}$ and $(x_\alpha)_{\alpha\in A}$ and $(y_\beta)_{\beta\in B}$ be nets in $P$ such that $x_\alpha\xrightarrow{o_i} x\in P$ and $y_\beta\xrightarrow{o_i} y\in P$.
	If for every $\alpha_0\in A$ and $\beta_0\in B$ there are $\alpha\in A_{\geq\alpha_0}$ and $\beta\in B_{\geq\beta_0}$ such that $x_\alpha\leq y_\beta$, then $x\leq y$.
\end{proposition}

\begin{proof}
	By Proposition \ref{pro:basic_convergences} it is sufficient to show the statement for $i=3$. In this case, 	
	there are nets $(\hat{x}_\gamma)_{\gamma \in C}$,  $(\check{x}_\delta)_{\delta \in D}$,
	$(\hat{y}_\varepsilon)_{\varepsilon \in E}$,  $(\check{y}_\varphi)_{\varphi \in F}$	
	 in $P$ and maps 
	$\eta_x\colon C \times D \rightarrow A$, $\eta_y\colon E \times F \rightarrow B$
		 such that $\hat{x}_\gamma \uparrow x$,
	$\check{x}_\delta \downarrow x$, $\hat{y}_\varepsilon \uparrow y$,
	$\check{y}_\varphi \downarrow y$,
		  $\hat{x}_\gamma \leq x_\alpha \leq \check{x}_\delta$, $\hat{y}_\varepsilon \leq y_\beta \leq \check{y}_\varphi$ 		  
		  for every 
	$\gamma \in C$, $\delta\in D$, $\varepsilon\in E$, $\varphi\in F$, $\alpha \in A_{\geq \eta_x(\gamma,\delta)}$ and $\beta \in B_{\geq \eta_y(\varepsilon,\varphi)}$.
	
	For every $\gamma\in C$ and $\varphi\in F$ we have that $\hat{x}_\gamma\le \check{y}_\varphi$. Indeed, let $\delta\in D$, $\varepsilon\in E$ and note that by assumption there are $\alpha\in A_{\geq\eta_x(\gamma,\delta)}$ and $\beta\in B_{\geq\eta_y(\varepsilon,\varphi)}$ such that $\hat{x}_\gamma\leq x_\alpha\leq y_\beta\le \check{y}_\varphi$.
	From $\hat{x}_\gamma \uparrow x$ and 
	$\check{y}_\varphi \downarrow y$ we conclude that $x\leq y$.
	\end{proof}

\begin{remark}
	\label{rem:unique_order_limits}
	Note that Proposition \ref{pro:monotony} immediately implies the uniqueness of the $o_i$-limits. 
\end{remark}

The combination of Theorem \ref{thm:orderclosed} with Proposition \ref{pro:monotony} yields the following statement.
\begin{corollary}\label{cor:upperboundset_orderclosed}
	For every $p\in P$ the sets $P_{\leq p}$ and $P_{\geq p}$ are order closed.
\end{corollary}

\begin{remark}\label{rem:Floyd_sigma_comp}
Corollary \ref{cor:upperboundset_orderclosed} implies that for every $p\in P$ the set $\{p\}$ is order closed, thus $P$ with the order topology is $\operatorname{T_1}$. Note that the order topology is not Hausdorff, in general. Indeed, 
	a combination of Proposition \ref{pro:basic_convergences} and Remark \ref{rem:decreasingnet} yields that the order topology is always $\sigma$-compatible in the sense of  \cite{Floyd1955}. Thus, \cite[Theorem 1]{Floyd1955} presents an example of a complete Boolean algebra on which the order topology is not Hausdorff.
\end{remark}

The following statement is a generalisation of the sandwich theorem for sequences given in \cite[Chapter II, \S 6,c)]{Vulikh67}.

\begin{proposition} \label{pro:sandwichtheorem}
	\begin{itemize}
		\item[(i)] Let $(x_\alpha)_{\alpha \in A}$, $(y_\alpha)_{\alpha \in A}$ and $(z_\alpha)_{\alpha \in A}$ be nets in $P$ such that
		$x_\alpha \xrightarrow{o_1} p\in P$ and $z_\alpha \xrightarrow{o_1}p$.
		If for every $\alpha \in A$ one has $x_\alpha \leq y_\alpha \leq z_\alpha$, then $y_\alpha \xrightarrow{o_1}p$.
		\item[(ii)] Let $(x_\alpha)_{\alpha \in A}$, $(y_\alpha)_{\alpha \in A}$ and $(z_\alpha)_{\alpha \in A}$ be nets in $P$ such that
		$x_\alpha \xrightarrow{o_2} p\in P$ and $z_\alpha \xrightarrow{o_2}p$.
		If there is $\alpha_0 \in A$ such that for each $\alpha \in A_{\geq \alpha_0}$ we have 
		$x_\alpha \leq y_\alpha \leq z_\alpha$, then
		$y_\alpha \xrightarrow{o_2}p$.
		\item[(iii)] Let $(x_\alpha)_{\alpha \in A}$, $(y_\beta)_{\beta \in B}$ and $(z_\gamma)_{\gamma \in C}$ be nets in $P$ such that
		$x_\alpha \xrightarrow{o_3} p\in P$ and $z_\gamma \xrightarrow{o_3}p$. If for $(\alpha_0,\gamma_0)\in A \times C$ there is
		$\beta_0 \in B$ such that for all $\beta \in B_{\geq \beta_0}$ there is
		$(\alpha,\gamma)\in A_{\geq \alpha_0}\times C_{\geq \gamma_0}$ with
		$x_\alpha \leq y_\beta \leq z_\gamma$, then
		$y_\beta \xrightarrow{o_3}p$.
	\end{itemize} 			 	
\end{proposition}			

\begin{proof} To show (i), let $x_\alpha \xrightarrow{o_1} p\in P$ and $z_\alpha \xrightarrow{o_1}p$. Thus there are nets $(\hat{x}_\alpha)_{\alpha \in A}$ and $(\check{z}_\alpha)_{\alpha \in A}$ in $P$ such that $\hat{x}_\alpha \uparrow p$, $\check{z}_\alpha \downarrow p$ and $\hat{x}_\alpha \leq x_\alpha \leq y_\alpha \leq z_\alpha \leq \check{z}_\alpha$ for every $\alpha \in A$, hence we obtain $y_\alpha \xrightarrow{o_1}p$. The proof of (ii) is similar. 
	
To show (iii), assume $x_\alpha \xrightarrow{o_3} p\in P$ and $z_\gamma \xrightarrow{o_3}p$. Hence there are nets $(\hat{x}_\delta)_{\delta \in D}$, $(\check{x}_\kappa)_{\kappa \in K}$, $(\hat{z}_\lambda)_{\lambda \in L}$ and $(\check{z}_\epsilon)_{\epsilon \in E}$ in $P$ and maps $\eta_x \colon D \times K\rightarrow A$ and $\eta_z \colon L \times E \rightarrow C$ such that $\hat{x}_\delta \uparrow p$, $\check{x}_\kappa \downarrow p$, $\hat{z}_\lambda \uparrow p$, $\check{z}_\epsilon \downarrow p$, $\hat{x}_\delta\leq x_\alpha \leq \check{x}_\kappa$ for all $(\delta,\kappa)\in D\times K$ and $\alpha \in A_{\geq \eta_x(\delta,\kappa)}$, and $\hat{z}_\lambda \leq z_\gamma \leq \check{z}_\epsilon$ for all $(\lambda,\epsilon)\in L\times E$ and $\gamma \in C_{\geq \eta_z(\lambda,\epsilon)}$.
Fix $\kappa \in K$ and $\lambda \in L$. By assumption, for $(\delta,\epsilon)\in D \times E$ there is $\beta_{(\delta,\epsilon)}\in A$ such that for all $\beta \in B_{\geq \beta_{(\delta,\epsilon)}}$ there exists $(\alpha,\gamma)\in A_{\geq \eta_x(\delta,\kappa)} \times C_{\geq \eta_z(\lambda,\epsilon)}$ with $x_\alpha\leq y_\beta \leq z_\gamma$, hence also $\hat{x}_\delta\leq x_\alpha \leq y_\beta \leq z_\gamma\leq \check{z}_\epsilon$. Thus $\eta_y \colon D \times E \rightarrow B$ with $\eta_y(\delta,\epsilon):=\beta_{(\delta,\epsilon)}$ defines a map such that $\hat{x}_\delta \leq y_\beta \leq \check{z}_\epsilon$ holds for every $(\delta,\epsilon)\in D \times E$ and $\beta \in B_{\geq \eta_y(\delta,\epsilon)}$. This proves $y_\beta \xrightarrow{o_3}p$.
\end{proof}

If all three nets have the same index set, we can simplify (iii) to the 
statements given in the following Corollary. 

\begin{corollary} 
	\label{cor:sandwichtheorem}
	Let $(x_\alpha)_{\alpha \in A}$, $(y_\alpha)_{\alpha \in A}$ and $(z_\alpha)_{\alpha \in A}$ be nets in $P$ such that
	$x_\alpha \xrightarrow{o_3} p\in P$ and $z_\alpha \xrightarrow{o_3}p$.
	\begin{itemize}
		\item[(i)] If there is $\delta \in A$ such that for each $\alpha \in A_{\geq \delta}$ we have 
		$x_\alpha \leq y_\alpha \leq z_\alpha$, then $y_\alpha \xrightarrow{o_3}p$.
		\item[(ii)] If for every $\delta\in A$ there is is $\alpha_{\delta} \in A$ such that 
		for every $\alpha \in A_{\geq \alpha_{\delta} }$ we have
		$x_{\delta}  \leq y_\alpha \leq z_{\delta}$,
		then $y_\alpha \xrightarrow{o_3}p$.
	\end{itemize}
\end{corollary}

\begin{proof} 
	For $(\alpha_0,\gamma_0)\in A\times A$ there is $\beta_0 \in A$ with 
	$\beta_0\geq \delta$, $\beta_0\geq \alpha_0$ and $\beta_0\geq \gamma_0$. 
	For $\beta \in A_{\geq \beta_0}$ the inequality
	$x_\beta \leq y_\beta \leq z_\beta$ is valid. If we set  $\alpha:=\beta$ and $\gamma:=\beta$, we obtain
	$(\alpha,\gamma)\in A_{\geq \alpha_0}\times C_{\geq \gamma_0}$ with $x_\alpha= x_\beta \leq y_\beta \leq z_\beta =z_\gamma$.	Hence Proposition \ref{pro:sandwichtheorem}(iii) implies the statement (i). 
	
	For $(\alpha_0,\gamma_0)\in A\times A$ there is $\beta_0 \in A$ with 
	$\beta_0\geq \alpha_0$ and $\beta_0\geq \gamma_0$. 
	Now the assumption
	implies the existence of $\alpha_{\beta_0}\in A$ with $x_{\beta_0}\leq y_\beta \leq z_{\beta_0}$
	for every $\beta \in A_{\geq \alpha_{\beta_0}}$. 
	For $\beta \in A_{\geq \beta_0}$ we set $\alpha:=\beta_0$
	and $\gamma:=\beta_0$ to get $(\alpha,\beta)\in A_{\geq \alpha_0}\times A_{\geq \gamma_0}$
	with $x_\alpha= x_{\beta_0} \leq y_\beta \leq z_{\beta_0}= z_\gamma$. Hence Proposition \ref{pro:sandwichtheorem}(iii) implies the statement (ii) as well. 		
\end{proof}

In distributive lattices the lattice operations are compatible with the order convergences. 

\begin{proposition} \label{pro:inf_o_i}
	Let $P$ be a distributive lattice and let $(x_\alpha)_{\alpha\in A}$ and $(y_\beta)_{\beta\in B}$ be nets in $P$. Let $A\times B$ be ordered component-wise and let $i\in\{1,2,3\}$. If $x_\alpha\xrightarrow{o_i} x\in P$ and $y_\beta\xrightarrow{o_i} y\in P$, then  the net $(x_\alpha\wedge y_\beta)_{(\alpha,\beta)\in A\times B}$ satisfies  $x_\alpha\wedge y_\beta\xrightarrow{o_i} x\wedge y$. 
	An analogous statement is valid for the supremum.
\end{proposition}

\begin{proof}
	We show the result for $i=1$; the cases $i=2$ and $i=3$ are similar.
	Let $(\hat{x}_\alpha)_{\alpha\in A}$,  $(\check{x}_\alpha)_{\alpha\in A}$, 
	$(\hat{y}_\beta)_{\beta\in B}$  and $(\check{y}_\beta)_{\beta\in B}$ be nets in $P$ such that $\hat{x}_\alpha\uparrow x$, $\check{x}_\alpha\downarrow x$,
	$\hat{y}_\beta\uparrow y$, $\check{y}_\beta\downarrow y$, $\hat{x}_\alpha\leq x_\alpha\leq \check{x}_\alpha$ for every $\alpha\in A$, and $\hat{y}_\beta\leq y_\beta\leq \check{y}_\beta$ for every $\beta\in B$.
	We get immediately that 
	$\hat{x}_\alpha\wedge\hat{y}_\beta\leq x_\alpha\wedge y_\beta\leq \check{x}_\alpha\wedge \check{y}_\beta$ for every $(\alpha,\beta)\in A\times B$ and that
	the net $\left(\check{x}_\alpha\wedge \check{y}_\beta\right)_{(\alpha,\beta)\in A\times B}$ satisfies  $\check{x}_\alpha\wedge \check{y}_\beta\downarrow x\wedge y$. Furthermore, \eqref{equ:distr_law} with $M=\{\hat{x}_\alpha;\, \alpha\in A\}$ and $N=\{\hat{y}_\beta;\,\beta\in B\}$ implies $\hat{x}_\alpha\wedge \hat{y}_\beta\uparrow x\wedge y$.
\end{proof}
\begin{remark}\label{rem:subnet_o_i}
	Let $(x_\alpha)_{\alpha\in A}$ be a net in $P$ and let $(y_\beta)_{\beta\in B}$ be a subnet of $(x_\alpha)_{\alpha\in A}$. Let $x\in P$ and fix $i\in\{1,2,3\}$.
	If 	$x_\alpha \xrightarrow{o_i} x$, then $y_\beta\xrightarrow{o_i} x$. 
	This will be useful in combination with the following statement. 	
	Let $Q$ be a partially ordered set. For a net $(x_\alpha)_{\alpha\in A}$ in $P$ and $(y_\alpha)_{\alpha\in A}$  in $Q$ and a map $f\colon P\times Q\to Q$ the net $(f(x_\alpha,y_\alpha))_{\alpha\in A}$
is a subnet of $(f(x_\alpha,y_\beta))_{(\alpha,\beta)\in A\times A}$.

In particular, if $(x_\alpha)_{\alpha 
	\in A}$ and $(y_\alpha
)_{\alpha \in A}$ are nets in a distributive lattice $P$ with $x_\alpha \xrightarrow{o_i} x\in P$ and $y_\alpha
\xrightarrow{o_i} y\in P$, then Proposition \ref{pro:inf_o_i}
shows that the net $(x_\alpha
\wedge y_\alpha)_{\alpha
	\in A}$ satisfies $x_\alpha \wedge y_\alpha \xrightarrow{o_i} x\wedge y$. 

This technique will also be applied to the addition of nets in partially ordered abelian groups and the multiplication of a scalar net and a net in a partially ordered vector space in the subsequent discussion. 
\end{remark}
\section{Continuous maps on partially ordered sets}
In this section, $P$ and $Q$ are partially ordered sets. For $o_1$-, $o_2$-, $o_3$- and $\tau_o$-convergence, we will introduce the corresponding concepts of continuity. It will be shown that for monotone maps these concepts are equivalent. 

\begin{definition} 
	\label{def:ordercontinuity}
	A map $f\colon P\to Q$ is called
	\begin{itemize}
		\item[(i)] \emph{$o_i$-continuous in $x\in P$}, if for every net $(x_\alpha)_{\alpha\in A}$ with $x_\alpha\xrightarrow{o_i}x$ we have that $f(x_\alpha)\xrightarrow{o_i}f(x)$ (where $i\in \{1,2,3\}$). 
		\item[(ii)] \emph{order continuous in $x\in P$}, if it is continuous in $x$ with respect to the order topologies $\tau_o(P)$ and $\tau_o(Q)$, respectively.
	\end{itemize}
	$f$ is called \emph{$o_i$-continuous} (\emph{order continuous}, respectively) if it is $o_i$-continuous (order continuous, respectively) in $x$ for every $x\in P$.
	\end{definition}

\begin{theorem}\label{thm:ordercontinuous}
	Let $i\in \{1,2,3\}$.
	Every $o_i$-continuous map  $f\colon P\to Q$ is order continuous. 
\end{theorem}

\begin{proof}
	We show that for every order closed set $C\subseteq Q$ the preimage $[C]f$ is order closed in $P$.
			Indeed, let $C\subseteq Q$ be order closed. By Theorem \ref{thm:orderclosed} it suffices to show that for every net $(x_\alpha)_{\alpha\in A}$ in $[C]f$ with $x_\alpha\xrightarrow{o_i}x\in P$ we have that $x\in [C]f$.
			Since $f$ is $o_i$-continuous, we obtain $f(x_\alpha)\xrightarrow{o_i}f(x)$. Since $(f(x_\alpha))_{\alpha\in A}$ is a net in $C$ and $C$ is order closed, Theorem \ref{thm:orderclosed} implies that $f(x)\in C$, hence $x\in [C]f$.
	\end{proof}
	
To show that all concepts introduced in Definition \ref{def:ordercontinuity} coincide for monotone maps, we need the following lemma. 

\begin{lemma}\label{lem:topologicalconv_inf}
	Let $(x_\alpha)_{\alpha\in A}$ be a net in $P$ with $x_\alpha\xrightarrow{\tau_o}x\in P$. 
	\begin{itemize} \item[(i)]  If $\inf\{x_\alpha;\alpha\in A\}$ exists, then $\inf\{x_\alpha;\alpha\in A\}\leq x$. 
	\item[(ii)] If for every $\alpha\in A$ we have $x_\alpha\in P_{\geq x}$, then $\inf\{x_\alpha;\alpha\in A\}$ exists and satisfies $\inf\{x_\alpha;\alpha\in A\}=x$.
\end{itemize}
\end{lemma}

\begin{proof}
	Note that for both statements it is sufficient to show that for every lower bound $p$ of 
	$\{x_\alpha;\alpha\in A\}$ we have 
	$p\leq x$.
	
	Let $p$ be a lower bound of 
	$\{x_\alpha;\alpha\in A\}$, i.e.\ for every $\alpha\in A$ we have $x_\alpha\in P_{\geq p}$. Since  $x_\alpha\xrightarrow{\tau_o}x$ and  $P_{\geq p}$ is order closed by Corollary \ref{cor:upperboundset_orderclosed}, we conclude $x\in P_{\geq p}$, i.e.\ $p\leq x$.
\end{proof}

\begin{theorem} \label{thm:monotone_ordercont}
	Let $f\colon P\to Q$ be a monotone map and $i\in\{1,2,3\}$. Then the following statements are equivalent:
	\begin{itemize}
		\item[(i)] $f$ is $o_i$-continuous.
		\item[(ii)] $f$ is order continuous.
		\item[(iii)] For every net $(x_\alpha)_{\alpha\in A}$ in $P$ and $x\in P$ the following implications are valid: 
		\begin{itemize}
			\item[(a)] 
		If $x_\alpha\downarrow x$ then $\inf\{f(x_\alpha);\alpha\in A\}$ exists and satisfies $\inf\{f(x_\alpha);\alpha\in A\}=f(x)$.
		\item[(b)]
		If $x_\alpha\uparrow x$ then $\sup\{f(x_\alpha);\alpha\in A\}$ exists and satisfies $\sup\{f(x_\alpha);\alpha\in A\}=f(x)$.
			\end{itemize}
	\end{itemize} 
\end{theorem}

\begin{proof}
	The implication (i)$\Rightarrow$(ii) is contained in Theorem \ref{thm:ordercontinuous}. We show (ii)$\Rightarrow$(iii). Let $(x_\alpha)_{\alpha\in A}$ be a net in $P$ such that $x_\alpha\downarrow x\in P$. Due to Remark \ref{rem:decreasingnet} and Proposition \ref{pro:basic_convergences} this implies $x_\alpha\xrightarrow{\tau_o} x$. Since $f$ is order continuous, we obtain $f(x_\alpha)\xrightarrow{\tau_o} f(x)$. Furthermore, the monotony of $f$ yields for every $\alpha\in A$ that $f(x_\alpha)\in Q_{\geq f(x)}$. Thus
	Lemma \ref{lem:topologicalconv_inf} (ii) implies that $\inf\{f(x_\alpha);\alpha\in A\}$ exists and satisfies  $\inf\{f(x_\alpha);\alpha\in A\}= f(x)$. The second statement in (iii) is shown analogously.
	
	It remains to show (iii)$\Rightarrow$(i). We proof this implication for $i=3$; the argumentation for $i\in\{1,2\}$ is similar. Let $(x_\alpha)_{\alpha\in A}$ be a net such that $x_\alpha\xrightarrow{o_3}x\in P$, i.e.\ there are nets $(\hat{x}_\beta)_{\beta \in B}$
	and $(\check{x}_\gamma)_{\gamma \in C}$ in $P$ and a map 
	$\eta\colon B \times C \rightarrow A$ such that $\hat{x}_\beta \uparrow x$,
	$\check{x}_\gamma \downarrow x$ and $\hat{x}_\beta \leq x_\alpha \leq \check{x}_\gamma$ for every 
	$\beta\in B$, $\gamma \in C$ and $\alpha \in A_{\geq \eta(\beta,\gamma)}$. The monotony of $f$ and condition (iii) implies that $f(\hat{x}_\beta) \uparrow f(x)$ and 
	$f(\check{x}_\gamma) \downarrow f(x)$.
	Furthermore the monotony of $f$ yields $f(\hat{x}_\beta) \leq f(x_\alpha) \leq f(\check{x}_\gamma)$ for every 
	$\beta\in B$, $\gamma \in C$ and $\alpha \in A_{\geq \eta(\beta,\gamma)}$. Thus $f(x_\alpha)\xrightarrow{o_3}f(x)$. 
\end{proof}

Combining Theorem \ref{thm:monotone_ordercont} and Proposition \ref{pro:infimum} we obtain the following statement.
\begin{corollary} \label{coro:orderembeddingswithroderdenseimagesarecontinuous}
	Every order embedding $f\colon P\to Q$ for which $f[P]$ is order dense in $Q$ is order continuous (and, hence, $o_i$-continuous, where $i\in\{1,2,3\}$).
\end{corollary}

\begin{remark}\label{rem:RestrictionandExtensionproperty}
	Assume that $M\subseteq P$ is order dense in $P$. Then
	the embedding $f\colon M\to P$ is order continuous 
	by Corollary 
	\ref{coro:orderembeddingswithroderdenseimagesarecontinuous}, therefore the induced topology of $\tau_o(P)$ on $M$ satisfies
	\begin{equation} \label{equ:restrictionproperty}
	\{O\cap M; O\in\tau_o(P)\}\subseteq \tau_o(M).
	\end{equation}			 
	Thus for every order closed set $N\subseteq P$ 
	we obtain that $N \cap M$ is order closed in $M$. By means of Theorem \ref{thm:orderclosed} this generalises \cite[Proposition 5.1(iii)]{IaB}. Example \ref{exa:extensionprop} below shows that the converse implication in \eqref{equ:restrictionproperty} is not valid, in general. 
\end{remark}

The next statement follows from Proposition \ref{pro:o2_o3_convergenceDedekind_complete_lattice}.

\begin{proposition}
	Let $f \colon P \to Q$ be a map. \begin{itemize}
		\item[(i)] If $P$ is a Dedekind complete lattice and $f$ is $o_2$-continuous, then $f$ is also $o_3$-continuous.
		\item[(ii)] If $Q$ is a Dedekind complete lattice and $f$ is $o_3$-continuous, then $f$ is also $o_2$-continuous.
	\end{itemize}
\end{proposition}

\begin{remark}
	(i) Note that by Remark \ref{rem:o_1_and_o_2} every $o_1$-continuous map is $o_2$-continuous. The  converse implication is not true, in general, see Example \ref{exa:ordercontnotob} below, but it is open whether it is true in partially ordered abelian groups. 
		
	(ii) In \cite[Example 1.8]{Abra} it is shown that $o_3$-continuity of maps between vector lattices does not imply $o_2$-continuity, in general. In	Corollary \ref{cor:orderboundedando2contiso3cont} below we present a setting where $o_2$-continuity implies $o_3$-continuity. It is an open question whether this implication is valid in more general situations. Moreover it is not clear under which conditions the converse implications in 
	Theorem \ref{thm:ordercontinuous}
	are true.
	
	(iii) In Theorem \ref{the:ogasawara_spaces_are_equal} we will present a situation where all concepts introduced in Definition \ref{def:ordercontinuity} coincide. 
\end{remark}

In \cite[Proposition 1.5]{Abra} it is shown that the $o_3$-convergence in a vector lattice $X$ is equivalent to the $o_2$-convergence in the Dedekind completion $X^{\delta}$ of $X$. To show that a generalisation\footnote{To link our notions with the one in \cite{Abra}, use Proposition \ref{pro:char_o_i_poag} below.} to lattices holds, we need the following technical statement.

	\begin{lemma}\label{lem:index_net}
		Let $P$ be a lattice, $Q$ a partially ordered set
		and $f\colon P \rightarrow Q$ an order embedding such that $f[P]$ is order dense in $Q$. 
		Let $(\check{y}_\alpha)_{\alpha\in A}$ be a net in $Q$ such that $\check{y}_\alpha \downarrow f(x)$ for $x \in P$. 
		If 
		\begin{align*}
		B:=\{v \in P;\,\exists \alpha \in A\colon f(v) \geq \check{y}_\alpha\}		
		\end{align*}
		is equipped with the reversed order of $P$, then $B$ is directed and
		$\inf B=x$. Thus $\check{x}_\beta:=\beta$ for all $\beta \in B$ defines a net 
		in $P$ with $\check{x}_\beta\downarrow x$.
		\end{lemma}
		\begin{proof} For $v_1,v_2\in B$ there are $\alpha_1,\alpha_2\in A$ such that 
		$f(v_1) \geq \check{y}_{\alpha_1}$ and $f(v_2)\geq \check{y}_{\alpha_2}$. Since $A$ is directed there is 
		$\alpha \in A$ with $ \alpha \geq \{\alpha_1,\alpha_2\}$. We use $\check{y}_\alpha \downarrow$ and get
		$f(v_1)\geq\check{y}_{\alpha}$ and $f(v_2)\geq \check{y}_{\alpha}$. By Proposition \ref{pro:infimum} we conclude 
		$f(v_1 \wedge v_2)=f(v_1)\wedge f(v_2)\geq \check{y}_\alpha$. Thus $v_1\wedge v_2\in B$, and we have shown $B$ to be directed.\\
		It is left to show that $\inf B=x$. For $v \in B$ we have $f(v)\geq \check{y}_\alpha \geq f(x)$ for some 
		$\alpha \in A$. Since $f$
		is order reflecting we know $x$ to be a lower bound of $B$. In order to show
		that $x$ is the greatest lower bound of $B$ let $z\in P$ be another lower bound. 
		For $\alpha \in A$ the monotony of $f$ implies 
		\begin{align*}
		f(z)\leq f[B]\supseteq f[\{v \in P;\, f(v)\geq \check{y}_\alpha\}]=\{y \in f[P];\, y \geq \check{y}_\alpha\}.
		\end{align*}
		Since $f[P]$ is order dense in $Q$ we conclude $f(z)\leq \inf\{y \in f[P];\, y \geq \check{y}_\alpha\}=\check{y}_\alpha$. Thus $\check{y}_\alpha \downarrow f(x)$ yields $f(z)\leq f(x)$. Since $f$ is order reflecting we 
		conclude $z\leq x$. This proves $x$ to be the greatest lower bound of $B$.
	\end{proof}
	
	\begin{proposition}
	\label{pro:Dedekindcompletionando3o2}
	Let $Q$ be a partially ordered set and $f\colon P \to Q$ an order embedding such that $f[P]$ is order dense in $Q$.
	Let $(x_\alpha)_{\alpha \in A}$ be a net in $P$ and $x\in P$.	
	\begin{itemize}
		\item[(i)] If $Q$ is a Dedekind complete lattice, then $x_\alpha \xrightarrow{o_3}x$ implies $f(x_\alpha) \xrightarrow{o_2} f(x)$.
		\item[(ii)] If $P$ is a lattice, then $f(x_\alpha) \xrightarrow{o_2} f(x)$ implies $x_\alpha \xrightarrow{o_3}x$.
	\end{itemize}
\end{proposition}

\begin{proof}
	To show (i), let $x_\alpha \xrightarrow{o_3} x$. Corollary \ref{coro:orderembeddingswithroderdenseimagesarecontinuous} implies $f(x_\alpha) \xrightarrow{o_3}f(x)$. Thus Proposition \ref{pro:o2_o3_convergenceDedekind_complete_lattice} yields $f(x_\alpha)\xrightarrow{o_2}f(x)$. 
	
	To prove (ii), let $f(x_\alpha) \xrightarrow{o_2} f(x)$. Hence there are nets $(\hat{y}_\alpha)_{\alpha \in A}$ and $(\check{y}_\alpha)_{\alpha \in A}$ with 
	$\hat{y}_\alpha \uparrow f(x)$, $\check{y}_\alpha \downarrow f(x)$ and $\hat{y}_\alpha \leq f(x_\alpha)\leq \check{y}_\alpha$ for all $\alpha \in A$.
	Let $(\check{x}_\beta)_{\beta \in B}$ be defined as in Lemma \ref{lem:index_net} and note that $\check{x}_\beta \downarrow x$. By the definition of $B$, for $\beta \in B$ there is $\alpha_\beta \in A$ such that $f(x_\alpha) \leq \check{y}_\alpha \leq \check{y}_{\alpha_\beta}\leq f(\beta)=f(\check{x}_\beta)$ for all $\alpha \in A_{\geq \alpha_\beta}$. Since $f$ is order reflecting we obtain $x_\alpha \leq \check{x}_\beta$. An analogous construction shows the existence of a net $(\hat{x}_\gamma)_{\gamma \in C}$ with $\hat{x}_\gamma \uparrow x$ and such that for $\gamma \in C$ there exists $\alpha_\gamma\in A$ with 
	$\hat{x}_\gamma \leq x_\alpha$ for all $\alpha \in A_{\geq \alpha_\gamma}$. For $(\beta,\gamma)\in B \times C$ let $\alpha_{(\beta,\gamma)} \in A$ be 
	such that $\alpha_{(\beta,\gamma)}\geq \alpha_\beta$ and $\alpha_{(\beta,\gamma)}\geq \alpha_\gamma$. Thus $\eta\colon B \times C \rightarrow A$, $(\beta,\gamma)\mapsto\alpha_{(\beta,\gamma)}$ yields a map as in the definition of the $o_3$-convergence.
\end{proof}

Proposition \ref{pro:Dedekindcompletionando3o2} in combination with Remark \ref{rem:o1o2}(a) yields the following.

\begin{corollary}
	\label{cor:Dedekindcompletionando3o1_orderbddnets}
	Let $Q$ be a partially ordered set that is directed upward and downward, and $f\colon P \to Q$ an order embedding such that $f[P]$ is order dense in $Q$.
	Let $(x_\alpha)_{\alpha \in A}$ be a net in $P$ such that $\{f(x_\alpha);\, \alpha \in A\}$ is bounded, and let $x\in P$.	
	\begin{itemize}
		\item[(i)] If $Q$ is a Dedekind complete lattice, then $x_\alpha \xrightarrow{o_3}x$ implies $f(x_\alpha) \xrightarrow{o_1} f(x)$.
		\item[(ii)] If $P$ is a lattice, then $f(x_\alpha) \xrightarrow{o_1} f(x)$ implies $x_\alpha \xrightarrow{o_3}x$.
	\end{itemize}
\end{corollary}

\begin{remark} 
	Note that the implications in Proposition \ref{pro:Dedekindcompletionando3o2}(ii) and in Corollary \ref{cor:Dedekindcompletionando3o1_orderbddnets}(ii) are not valid, in general. In Example \ref{exa:extensionprop} below a partially ordered vector space $P=X$ and a vector lattice $Q=Y$ are provided which lead to a counterexample, where $f\colon P\to Q$ is the inclusion map. 
\end{remark}

One can characterise $o_3$-convergence in lattices by means of $o_3$-convergence in a cover.

\begin{proposition}
	\label{pro:o3conv_inlatticeandDedeindcmpletion}
	Let $P$ be a lattice, let $Q$ be a partially ordered set and let $f\colon P \to Q$ be an order embedding such that $f[P]$ is order dense in $Q$. Let $(x_\alpha)_{\alpha \in A}$ be a net in $P$ and $x\in P$. Then $x_\alpha \xrightarrow{o_3}x$ if and only if $f(x_\alpha)\xrightarrow{o_3} f(x)$.
\end{proposition}

\begin{proof}
	If $x_\alpha \xrightarrow{o_3}x$, then $f(x_\alpha) \xrightarrow{o_3}f(x)$ in $f[P]$, hence also in $Q$. To show the converse implication, let $Q^\mu$ be the Dedekind-MacNeille completion\footnote{If $Q$ is a partially ordered set, then there is a complete lattice $Q^\mu$ and an order embedding $J\colon Q \to Q^\mu$ such that $J[Q]$ is order dense in $Q^\mu$. The set $Q^\mu$ is called \emph{Dedekind-MacNeille completion} of $Q$.} and $J\colon Q\to Q^\mu$ the canonical embedding. If $f(x_\alpha) \xrightarrow{o_3}f(x)$ in $Q$, then Proposition \ref{pro:Dedekindcompletionando3o2}(i) shows $J(f(x_\alpha))\xrightarrow{o_2} J(f(x))$. Since $J\circ f[P]$ is order dense in $J[Q]$ and $J[Q]$ is order dense in $Q^\mu$, by Proposition \ref{pro:orderdensitytransitive} we conclude $J\circ f[P]$ to be order dense in $Q^\mu$. Note furthermore that $J\circ f\colon P \to Q^\mu$ is an order embedding. 
	Hence Proposition \ref{pro:Dedekindcompletionando3o2}(ii) shows $x_\alpha \xrightarrow{o_3}x$.
\end{proof}

\begin{remark}
	In \cite[Example 1.4]{Abra} an example of a vector lattice $X$ and a net $(x_\alpha)_{\alpha \in A}$ with $\{x_\alpha;\, \alpha \in A\}$ bounded is given that $o_3$-convergences, but does not $o_2$-converge. Hence by Proposition \ref{pro:basic_convergences} the net $(x_\alpha)_{\alpha \in A}$ does not $o_1$-converge. Since $(x_\alpha)_{\alpha \in A}$ is $o_3$-convergent in $X$ and $\{x_\alpha;\, \alpha \in A\}$ is bounded, Corollary \ref{cor:Dedekindcompletionando3o1_orderbddnets} implies $(x_\alpha)_{\alpha \in A}$ to be $o_1$-convergent in $X^\delta$, and hence $o_2$-convergent in $X^\delta$.
	Thus an analogue of Proposition \ref{pro:o3conv_inlatticeandDedeindcmpletion} for  $o_1$-convergence and  $o_2$-convergence is not valid.	
	
	In Proposition \ref{pro:o3conv_inlatticeandDedeindcmpletion} the statement is not valid for arbitrary partially ordered sets $P$. Indeed, in Example \ref{exa:extensionprop} below we will present a partially ordered vector space $P=X$, a vector lattice $Q=Y$, and a net $(x_\alpha)_{\alpha \in A}$ in $P$ such that for the canonical embedding $f\colon P \to Q$ we have that $f(x_\alpha) \xrightarrow{o_3}f(x)$, but $(x_\alpha)_{\alpha \in A}$ does not $o_3$-converge.
\end{remark}

Next we discuss the link between $o_1$-continuity and order boundedness. The proof of the subsequent proposition is adopted from \cite[Proposition 149]{Mali2017}. 

\begin{proposition}\label{pro:o1ob}
Every $o_1$-continuous map  $f\colon P\to Q$ is order bounded.
\end{proposition}

\begin{proof}
	Let $A:=[v,w]$ be an order interval in $P$ and consider the net $(x_{\alpha})_{\alpha\in A}$ with $x_\alpha:=\alpha$. Note that $x_\alpha\uparrow w$, therefore $x_\alpha\xrightarrow{o_1}w$. Thus $f(x_\alpha)\xrightarrow{o_1}f(w)$, hence there are nets $(\hat{y}_\alpha)_{\alpha\in A}$ and $(\check{y}_\alpha)_{\alpha\in A}$ such that $\hat{y}_\alpha\uparrow f(w)$, $\check{y}_\alpha\downarrow f(w)$ and $\hat{y}_\alpha\leq  f(x_\alpha)\leq \check{y}_\alpha$ for every $\alpha\in A$. Consequently $f\left[[v,w]\right]\subseteq [\hat{y}_v, \check{y}_v]$. 
\end{proof}
	
The subsequent simple example 
shows that $o_2$-, $o_3$-, and order continuity do not imply order boundedness, in general. 	
	
	\begin{example}
		\label{exa:ordercontnotob}
		Consider the partially ordered set $P:=\mathbb{R}\setminus\{0\}$
		with the standard order and the map $f\colon P\to P$, $x\mapsto \frac{1}{x^2}$. Clearly, $f$ is not order bounded and, hence, not $o_1$-continuous due to Proposition \ref{pro:o1ob}. Since $f$ is continuous with respect to the standard topology of $P$, Example \ref{exa:opensubsetsofR} yields that $f$ is $o_2$-continuous, $o_3$-continuous and order continuous. 
	\end{example}

\section{Order convergence and order topology in partially ordered abelian groups} 
Let $G$ be a partially ordered abelian group. In this section, we characterise net catching sets  as well as the three concepts of order convergence in partially ordered abelian groups. 
 
\begin{proposition}\label{pro:char_netcatchingsetImhoff}
Let $U\subseteq G$ and $x\in U$.
\begin{itemize}
	\item[(i)] $U$ is a net catching set for $0$ if and only if for every net $(x_\alpha)_{\alpha\in A}$ in $G$ with $x_\alpha\downarrow 0$ there is $\alpha\in A$ such that $[-x_\alpha,x_\alpha]\subseteq U$.
	\item[(ii)] $U$ is a net catching set for $x$ if and only if $U-x$ is a net catching set for $0$.
\end{itemize}   	
\end{proposition}
\begin{proof}
	(i) Let $U$ be a net catching set for $0$. If $(x_\alpha)_{\alpha\in A}$ is a net in $G$ with $x_\alpha\downarrow 0$, then $-x_\alpha\uparrow 0$, hence $[-x_\alpha,x_\alpha]\subseteq U$.
		
	For the converse implication, 
	we have to show that  $U$ is a net catching set for $0$. Let 
	$(\hat{x}_\alpha)_{\alpha\in A}$ and $(\check{x}_\alpha)_{\alpha\in A}$ be nets in $G$
	with $\hat{x}_\alpha \uparrow 0$ and $\check{x}_\alpha\downarrow 0$. 
	Thus $(\check{x}_\alpha-\hat{x}_\alpha)\downarrow 0$. By the assumption there is 
	$\alpha \in A$ such that $[\hat{x}_\alpha,\check{x}_\alpha]\subseteq[-(\check{x}_\alpha-\hat{x}_\alpha), \check{x}_\alpha-\hat{x}_\alpha]\subseteq U$.	
	
	The result in (ii) follows from the fact that $x_\alpha\downarrow x$ if and only if $x_\alpha-x\downarrow 0$ (and the similar statement for increasing nets). 
\end{proof}

\begin{remark}	
	In the case of a partially ordered vector spaces, the concept of O-neighbourhood is introduced in \cite[Definition 3.3]{Imh}. Proposition  \ref{pro:char_netcatchingsetImhoff} shows that O-neighbourhoods are exactly the net catching sets.
\end{remark}

\begin{remark}\label{rem:G+-G+orderopen}
	\begin{itemize}
		\item[(a)] The set $G_+$ is order closed, due to Corollary \ref{cor:upperboundset_orderclosed}.\item[(b)] The set $G_+-G_+$ is order closed. Indeed, by Theorem \ref{thm:orderclosed} it is sufficient to show that $G_+-G_+$ is closed under $o_1$-convergence. Let $(x_\alpha)_{\alpha \in A}$ be a net in $G_+-G_+$ such that $x_\alpha \xrightarrow{o_1}x\in G$. Then there are nets $(\hat{x}_\alpha)_{\alpha\in A}$ and $(\check{x}_\alpha)_{\alpha \in A}$ such that $\hat{x}_\alpha \uparrow x$, $\check{x}_\alpha \downarrow x$ and $\hat{x}_\alpha\leq x_\alpha \leq \check{x}_\alpha$ for every $\alpha \in A$. Thus for every $\alpha \in A$ we obtain $x\in G_+ + \hat{x}_\alpha \subseteq G_+ +(x_\alpha-G_+) \subseteq G_+ + ((G_+-G_+) -G_+)=G_+-G_+$.
		
		\item[(c)] The set $G_+-G_+$ is order open. Indeed, by 
		Proposition \ref{pro:char_netcatchingsetImhoff} (ii)
		it is sufficient to show that   $G_+-G_+$ is a net-catching set for $0$. Let 
		$(x_\alpha)_{\alpha\in A}$ be a net in $G$ with $x_\alpha\downarrow 0$, 
		then for every $\alpha \in A$ we have
		$[-x_\alpha,x_\alpha]\subseteq x_\alpha-G_+ \subseteq G_+-G_+$.
	\end{itemize} 	
\end{remark}

 Note that for nets $(x_\alpha)_{\alpha\in A}$  and $(y_\beta)_{\beta \in B}$ in $G$ with $x_\alpha\downarrow x\in G$ and $y_\beta\downarrow y\in G$ the net $(x_\alpha+y_\beta)_{(\alpha,\beta)\in A\times B}$ satisfies $x_\alpha+y_\beta\downarrow x+y$, where $A\times B$ is ordered component-wise.
 This yields the following statement.
 
 \begin{proposition} \label{pro:plus_o_i}
 	Let $G$ be a partially ordered abelian group and let $(x_\alpha)_{\alpha\in A}$ and $(y_\beta)_{\beta\in B}$ be nets in $G$. Let $A\times B$ be ordered component-wise and let $i\in\{1,2,3\}$. If $x_\alpha\xrightarrow{o_i} x\in G$ and $y_\beta\xrightarrow{o_i} y\in G$, then the net $(x_\alpha+ y_\beta)_{(\alpha,\beta)\in A\times B}$ satisfies  $x_\alpha+ y_\beta\xrightarrow{o_i} x+ y$. 
 \end{proposition}
 
 \begin{remark}\label{rem:t_o_not_linear}
 	Due to Remark \ref{rem:Floyd_sigma_comp} the order topology is $\operatorname{T}_1 $ and $\sigma$-compatible, hence the assumptions in \cite[Theorem 3]{Floyd1955} are satisfied. Since the map $G \to G\colon g\mapsto -g$ is order continuous for every partially ordered abelian group $G$, by \cite[Corollary]{Floyd1955}
 	there is a 
 	Dedekind complete vector lattice $X$ endowed with the order topology with the property that the addition $X\times X\to X$, $(x,y)\mapsto x+y$ is not continuous, where $X\times X$ is equipped  with the product topology.
 	
 	As the order bound topology introduced in \cite[p.\ 20]{Namioka57} is always a linear topology, this shows that $\tau_o$ does not coincide with the order bound topology of $X$.
 \end{remark}

  The order convergences in vector lattices investigated in 
  \cite{Abra} are special cases of the $o_i$-convergences, as the next proposition shows. 

\begin{proposition} \label{pro:char_o_i_poag}
	Let 
	$(x_\alpha)_{\alpha \in A}$ be a net in $G$. Then 
	\begin{itemize}
		\item[(i)]
		$x_\alpha \xrightarrow{o_1} 0$  
		if and only if there is a net 
		$(\check{x}_\alpha)_{\alpha \in A}$ in $G$ such that $\check{x}_\alpha \downarrow 0$ and 
		$\pm x_\alpha \leq \check{x}_\alpha$ for every 
		$\alpha \in A$,
		\item[(ii)] 
		$x_\alpha \xrightarrow{o_2} 0$ 
		if and only if there is a net  
		$(\check{x}_\alpha)_{\alpha \in A}$ in $G$ and $\alpha_0 \in A$ such that $\check{x}_\alpha \downarrow 0$ and 
		$\pm x_\alpha \leq \check{x}_\alpha$ for every 
		$\alpha \in A_{\geq \alpha_0}$,
		\item[(iii)] 
		$x_\alpha \xrightarrow{o_3} 0$, 
		if and only if there is a net  $(\check{x}_\beta)_{\beta \in B}$
		and 
		a map 
		$\eta\colon B \rightarrow A$ such that $\check{x}_\beta \downarrow 0$ and
		$\pm x_\alpha \leq \check{x}_\beta$ for every 
		$\beta\in B$ and $\alpha \in A_{\geq \eta(\beta)}$,
		\item[(iv)] for every $i\in \{1,2,3\}$ and $x\in G$ we have that $x_\alpha\xrightarrow{o_i}x$ if and only if $x_\alpha-x\xrightarrow{o_i}0$. 
	\end{itemize}
\end{proposition}
\begin{proof}
	We show (iii), observe that (i) and (ii) are similar.
	Let $x_\alpha \xrightarrow{o_3} 0$. Then  
	Proposition \ref{pro:char_o3conv} yields the existence of nets   $(\hat{y}_\beta)_{\beta\in B}$ and $(\check{y}_\beta)_{\beta\in B}$ and a map $\eta\colon B\to A$ such that $\hat{y}_\beta\uparrow 0$,  $\check{y}_\beta\downarrow 0$ and $\hat{y}_\beta\leq x_\alpha\leq \check{y}_\beta$ for every $\beta\in B$ and $\alpha\in A_{\geq\eta(\beta)}$. For $\beta\in B$ define $\check{x}_\beta:=\check{y}_\beta-\hat{y}_\beta$. Observe that $\check{y}_\beta-\hat{y}_\beta\downarrow 0$. Furthermore $-\check{x}_\beta\leq \hat{y}_\beta \leq x_\alpha\leq \check{y}_\beta\leq \check{x}_\beta$ holds for all $\beta \in B$ and $\alpha \in A_{\geq \eta(\beta)}$. The converse implication in (iii) is straightforward. The statement in (iv) is a direct consequence of Proposition \ref{pro:plus_o_i}.
\end{proof} 

Order closed subgroups of lattice-ordered abelian groups are characterised as follows. 

\begin{proposition}\label{pro:McapG+oclosed}
	Let $M$ be a subgroup of a lattice-ordered abelian group $G$ such that $M$ is closed under the lattice operations of $G$ (i.e.\ for every $x,y\in M$ the element $x\vee y\in G$ belongs to $M$).	
	 Then $M$ is order closed if and only if $M\cap G_+$ is order closed.
\end{proposition}

\begin{proof}
	Let $M$ be order closed. Since $G_+$ is order closed, we obtain that $M\cap G_+$ is order closed.
	
	For the converse implication, we use Theorem \ref{thm:orderclosed}. Let $(x_\alpha)_{\alpha\in A}$ be a net in $M$ with $x_\alpha\xrightarrow{o_1}x\in G$. 
	By Proposition \ref{pro:inf_o_i} we obtain  $x_\alpha^+\xrightarrow{o_1}x^+$ and
	$x_\alpha^-\xrightarrow{o_1}x^-$. Since $x_\alpha^+, x_\alpha^-\in M\cap G_+$, we conclude $x=x^+-x^-\in M$.
\end{proof}

\section{The Riesz-Kantorovich formulas for group homomorphisms}
In this section we study conditions on partially ordered abelian groups $G$ and $H$ 
such that the set  $\operatorname{A}_{\operatorname{b}}(G,H)$ of all order bounded additive maps turns out to be a lattice-ordered abelian group. The arguments are straightforward adaptations of the classical Riesz-Kantorovich theorem, see \cite{Riesz30} and \cite{Kan1940}. We include the proofs here for sake of completeness. 

\begin{proposition}\label{pro:Kantorovich}
	Let $G$ and $H$ be partially ordered abelian groups such that $G$ is directed. Let $f\colon G_+\to H$ be a semigroup homomorphism. Then there exists a unique additive map $g\colon G\to H$ such that $f=g$ on $G_+$. Moreover, if $f[G_+]\subseteq H_+$, then $g$ is monotone.
	\end{proposition}
\begin{proof}
	First observe that for $u,v,x,y\in G_+$ with $v-u=y-x$ we have that $f(v)-f(u)=f(y)-f(x)$. Indeed, from $v+x=u+y$ it follows that $f(v)+f(x)=f(v+x)=f(u+y)=f(u)+f(y)$.
	
	For $x\in G$ there are $u,v\in G_+$ such that $x=u-v$. Define $g(x):=f(u)-f(v)$ and note that the definition is independent of the choice of $u$ and $v$. 
	
	$g$ is additive. Indeed, let $x,y\in G$ be such that $x=v-u$ and $y=z-w$ with $u,v,w,z\in G_+$. Since  
	$f(v)+f(z)+f(u+w)=f(v+z)+f(u)+f(w)$, we have
	\begin{align*}g(x+y)&=g(v-u+z-w)=f(v+z)-f(u+w)\\&=f(v)-f(u)+f(z)-f(w)=g(v-u)+g(z-w)\\&=g(x)+g(y).\end{align*}
	Moreover, $g$ is unique.
\end{proof}

The next proposition contains the crucial conditions under which the partially ordered abelian group $\operatorname{A}_{\operatorname{b}}(G,H)$ is a lattice.

\begin{proposition}\label{pro:RK}
	Let $G$  be a directed partially ordered abelian group with the Riesz decomposition property and let $H$ be a Dedekind complete lattice-ordered abelian group. For $f\in \operatorname{A}_{\operatorname{b}}(G,H)$ and $x\in G_+$ define
	\[g(x):=\sup\{f(u); \, u\in [0,x]\}.\]
	Then there exists a unique additive map $h\in \operatorname{A}_{+}(G,H)$ such that $h=g$ on $G_+$. Moreover, the supremum of $f$ and $0$ exists in $\operatorname{A}_{\operatorname{b}}(G,H)$ and equals $h$. 
\end{proposition}
\begin{proof}
	As $f$ is order bounded and $H$ is Dedekind complete, $g\colon G_+\to H_+$ is well-defined. 
	
	To show that $g$ is a semigroup homomorphism,
	let $x,y\in G_+$. For $u\in [0,x]$ and $v\in[0,y]$ we have $u+v\in [0,x+y]$ and $f(u+v)=f(u)+f(v)$, hence $g(x+y)\geq f(u)+f(v)$. Then, by taking the supremum over all $u$, we have $g(x+y)\geq g(x) +f(v)$. Similarly, the supremum over $v$ yields $g(x+y)\geq g(x) +g(y)$. Next, for $w\in[0,x+y]$ the Riesz decomposition property of $G$ provides us with $u\in[0,x]$ and $v\in[0,y]$ such that $w=u+v$. Then $f(w)=f(u)+f(v)\leq g(x)+g(y)$. The supremum over $w$ results in $g(x+y)\leq g(x)+g(y)$.
	
	
	According to Proposition \ref{pro:Kantorovich}, there exists $h\in \operatorname{A}_{+}(G,H)$ such that $h=g$ on $G_+$.
	
	Now we show that $h$ is the supremum of $f$ and $0$. Indeed, for $x\in G_+$ we have 
	$h(x)=g(x)\ge f(x)$, hence $h$ is an upper bound of $f$ and $0$. Let $q\in\operatorname{A}_{+}(G,H)$ be an upper bound of $f$. Then for $x\in G_+$ and $u\in[0,x]$ we have $q(x)\ge q(u)\ge f(u)$, so that $q(x)\ge g(x)=h(x)$, thus $q\ge h$. Hence $h=f\vee 0$. 
\end{proof}
In fact, Proposition \ref{pro:RK} yields the positive part $f^+:=h$ of $f$, hence  $\operatorname{A}_{\operatorname{b}}(G,H)$ is a lattice. 

\begin{theorem}\label{the:RK_final}
	Let $G$  be a directed partially ordered abelian group with the Riesz decomposition property and let $H$ be a Dedekind complete lattice-ordered abelian group.
	 Then $\operatorname{A}_{\operatorname{b}}(G,H)$ is a Dedekind complete lattice-ordered abelian group.
	 \end{theorem}
\begin{proof}
	It remains to show that
	$\operatorname{A}_{\operatorname{b}}(G,H)$ 	
	  is Dedekind complete. Let $A$ be a non-empty subset of $\operatorname{A}_{\operatorname{b}}(G,H)$ that is bounded from above. Let $q$ be an upper bound of $A$. Denote by $B$ the set of all suprema of finite non-empty subsets of $A$. 
	Note that $q$ is also an upper bound of $B$.
	For $x\in G_+$ define
	\begin{equation} \label{equ:InfRieszKantorovichB}
		g(x):=\sup\{f(x);\, f\in B\}.
	\end{equation} 
	To show that $g$ is a semigroup homomorphism, let $x,y\in G_+$. For every $f\in B$ we have $f(x+y)=f(x)+f(y)\le g(x)+g(y)$, hence $g(x+y)\le g(x)+g(y)$. Conversely, for every $f,h\in B$ we have $f\vee h\in B$, hence $g(x+y)\ge (f\vee h)(x+y)=(f\vee h)(x)+(f\vee h)(y)\ge f(x)+h(y)$. By taking supremum first over $f$ and then over $h$ we obtain $g(x+y)\ge g(x)+g(y)$. We conclude that $g$ is a semigroup homomorphism. 
	
	According to Proposition \ref{pro:Kantorovich} there exists a unique map 
	$h\in \operatorname{A}(G,H)$ with $h=g$ on $G_+$. From the definition of $g$ it is clear that $h$ is an upper bound of $B$, and hence of $A$. As $A$ is non-empty, there is $f\in A$ such that  $f\leq h$. Moreover, $h\leq q$, hence $h\in\operatorname{A}_{\operatorname{b}}(G,H)$.
	
	As $q$ is an arbitrary upper bound of $A$, it follows that $h$ is the supremum of $A$.
\end{proof}

\begin{remark}\label{rem:RK}
	Under the conditions of Theorem \ref{the:RK_final}, 
	the lattice operations in $\operatorname{A}_{\operatorname{b}}(G,H)$ are given by the following formulas.
	For every $x\in G_+$ and $f,g\in\operatorname{A}_{\operatorname{b}}(G,H)$ we have 
	\begin{eqnarray*}
		f^+(x)&=&\sup\{f(u); \, u\in [0,x]\},\\
		f^-(x)&=&\sup\{-f(u); \, u\in [0,x]\},\\
		|f|(x)&=&\sup\{|f(u)|; \, u\in [-x,x]\},\\
		(f\vee g)(x)&=&\sup\{f(x-u)+g(u); \, u\in [0,x]\},\\
		(f\wedge g)(x)&=&\inf\{f(x-u)+g(u); \, u\in [0,x]\}.
	\end{eqnarray*}
	These formulas are called the \emph{Riesz-Kantorovich formulas}.
\end{remark}

\begin{corollary}\label{cor:RK_pointwise_convergence}
	Under the conditions of Theorem \ref{the:RK_final} the following statements are valid.
	\begin{itemize}
		\item[(i)] If $A\subseteq \operatorname{A}_{\operatorname{b}}(G,H)$ is upward directed and bounded from above, then for every $x \in G_+$ we have
		\[(\sup A)(x)=\sup\{f(x); f \in A\}.\] 
		A similar statement is valid for the infimum of a downward directed set that is bounded from below.
		\item[(ii)] For a net  $(f_\alpha)_{\alpha\in A}$ in $\operatorname{A}_{\operatorname{b}}(G,H)$ we have $f_\alpha\downarrow 0$ if and only if for every $x\in G_+$ it holds $f_\alpha(x)\downarrow 0$.
		\item[(iii)] Let $i\in\{1,2,3\}$,   $(f_\alpha)_{\alpha\in A}$ be a net in $\operatorname{A}_{\operatorname{b}}(G,H)$ and $f\in \operatorname{A}_{\operatorname{b}}(G,H)$ with $f_\alpha \xrightarrow{o_i}  f$. Then   for every $x\in G$ one has $f_\alpha(x)\xrightarrow{o_i} f(x)$.
	\end{itemize} 
\end{corollary}

\begin{proof}
	To prove the statement in (i), let $B$ be as in the proof of Theorem \ref{the:RK_final}. Equation \eqref{equ:InfRieszKantorovichB} shows that for every $x \in G_+$ we have $(\sup A)(x)=\sup\{f(x); f \in B\}$. Since $A$ is a majorising subset of $B$, we obtain that $\{f(x); f \in A\}$ 
	is a majorising subset of $\{f(x); f \in B\}$. Thus we conclude $(\sup A)(x)=\sup\{f(x); f \in A\}$ by Lemma \ref{lem:majosetsandsuppe}. \\	
	The statement (ii) follows from (i).
	To show (iii), let the net $(\check{f}_\alpha)_{\alpha\in A}$ in $\operatorname{A}_{\operatorname{b}}(G,H)$ be such that $\pm(f_\alpha-f)\le \check{f}_\alpha\downarrow 0$. By (ii), for $x\in G_+$ we get $\pm(f_\alpha(x)-f(x))\le \check{f}_\alpha(x)\downarrow 0$. As $G$ is directed, Proposition
	\ref{pro:plus_o_i}
	 yields the statement for $x\in G$.
\end{proof}

\section{Properties of the set of order continuous homomorphisms of partially ordered abelian groups}
In this section let $G$, $H$ be partially ordered abelian groups. We show that under the conditions of the Riesz-Kantorovich Theorem \ref{the:RK_final} for an order bounded map $f\colon G \to H$ the four concepts of continuity from Section 5 coincide. We furthermore show that under the same conditions the set of order continuous maps is an order closed ideal in the lattice order abelian group $\operatorname{A}_{\operatorname{b}} (G,H)$ of all order bounded additive maps. For $i\in\{1,2,3\}$ we denote the set of all $o_i$-continuous maps in $\operatorname{A}_{\operatorname{b}}(G,H)$ by $\operatorname{A}^{o_i}_{\operatorname{b}}(G,H)$. The set of all order continuous maps in $\operatorname{A}_{\operatorname{b}}(G,H)$ is denoted by $\operatorname{A}^{\tau_o}_{\operatorname{b}}(G,H)$. Theorem \ref{thm:monotone_ordercont} reads then as
\begin{equation}
 \label{equ:def_cone_ocont_maps}
 \operatorname{A}^{o_i}_{\operatorname{b}}(G,H)\cap \operatorname{A}_+(G,H)=\operatorname{A}^{\tau_o}_{\operatorname{b}}(G,H)\cap \operatorname{A}_+(G,H)=:\operatorname{A}_+^{\operatorname{oc}}(G,H).\end{equation}
 
The set $\operatorname{A}_+^{\operatorname{oc}}(G,H)$ of positive order continuous additive maps is characterised as follows. 
 
 \begin{proposition}\label{pro:charAocplus}
 	For every $f\in \operatorname{A}(G,H)$ we have $f\in \operatorname{A}_+^{\operatorname{oc}}(G,H)$ if and only if for every net $(x_\alpha)_{\alpha\in A}$ with $x_\alpha\downarrow 0$ it holds $f(x_\alpha)\downarrow 0$. 
\end{proposition}	 
	 
	 \begin{proof}
	 	Let $f\in \operatorname{A}(G,H)$ be such that for every 
	 	net $(x_\alpha)_{\alpha\in A}$ with $x_\alpha\downarrow 0$ it holds $f(x_\alpha)\downarrow 0$.
	 	First we show that $f$ is monotone. Indeed, let $x\in G_+$, then for the net $(x_{\alpha})_{\alpha\in\{-x,0\}}$ with $x_{\alpha}=-\alpha$ we have $x_\alpha\downarrow 0$ and hence $f(x_\alpha)\downarrow 0$, which implies $f(x)=f(x_{-x})\ge 0$. 
	 	
	 	To show that $f$ is order continuous, note that the assumption implies that for every 
	 	net $(x_\alpha)_{\alpha\in A}$ with $x_\alpha\uparrow 0$ we have $f(x_\alpha)\uparrow 0$.   
	 	Then Theorem \ref{thm:monotone_ordercont} yields the order continuity of $f$, due to the translation invariance of infimum and supremum.
	 	 
	 	The converse implication follows directly from  
	 	Theorem \ref{thm:monotone_ordercont}.
	 \end{proof}
	 
As a consequence of Proposition \ref{pro:charAocplus} we obtain the following statement.

\begin{proposition} \label{pro:Aoc+oclosed}
	Under the conditions of Theorem \ref{the:RK_final},
	the set $\operatorname{A}_+^{\operatorname{oc}}(G,H)$ is order closed in $\operatorname{A}_{\operatorname{b}}(G,H)$.
\end{proposition}

\begin{proof}
	We use Theorem \ref{thm:orderclosed}.
		Let $(f_\alpha)_{\alpha\in A}$ be a net in $\operatorname{A}_+^{\operatorname{oc}}(G,H)$ such that $f_\alpha\xrightarrow{o_1}f\in \operatorname{A}_{\operatorname{b}}(G,H)$. By Remark \ref{rem:G+-G+orderopen} (a), the set $\operatorname{A}_{+}(G,H)$
	is order closed, hence  
	$f$ is monotone.
		By Proposition \ref{pro:char_o_i_poag} there is a net $(\check{f}_\alpha)_{\alpha\in A}$ such that $\check{f}_\alpha\downarrow 0$ and $\pm (f_\alpha-f)\le \check{f}_\alpha$ for every $\alpha\in A$. In order to apply Proposition \ref{pro:charAocplus}, let $(x_\beta)_{\beta\in B}$ be a net in $G$ such that $x_\beta\downarrow 0$. Since $f$ is monotone, $f(x_\beta)\downarrow$ and $0$ is a lower bound of $\{f(x_\beta);\, \beta\in B\}$. Let $z$ be a lower bound of $\{f(x_\beta);\, \beta\in B\}$. Let  $\beta\in B$. We will show that for every $\alpha\in A$ we have that $z\leq \check{f}_\alpha(x_\beta)$.
		Indeed, for $\gamma\in B_{\geq \beta}$ we calculate \[z\leq f(x_\gamma)\leq (f_\alpha+\check{f}_\alpha)(x_\gamma)\leq f_\alpha(x_\gamma)+\check{f}_\alpha(x_\beta),\]
		and from $f_\alpha \in\operatorname{A}_+^{\operatorname{oc}}(G,H)$ we conclude $\inf\{f_\alpha(x_\gamma);\, \gamma\in B_{\geq \beta} \}=0$. Hence $z\le \check{f}_\alpha(x_\beta)$.
		Thus, 
		Corollary \ref{cor:RK_pointwise_convergence} establishs $z\leq\inf\{\check{f}_\alpha(x_\beta);\, \alpha\in A\}=0$.     
\end{proof}

In order to establish $\operatorname{A}_{\operatorname{b}}^{\tau_o}(G,H)$ as an ideal in $\operatorname{A}_{\operatorname{b}}(G,H)$, we first show the following. 

\begin{proposition}\label{pro:full_subgroup}
	The set $\operatorname{A}^{o_i}_{\operatorname{b}}(G,H)$ is a full subgroup of $\operatorname{A}(G,H)$. 
\end{proposition}	 
\begin{proof}
	Due to Proposition \ref{pro:plus_o_i} the set $\operatorname{A}^{o_i}_{\operatorname{b}}(G,H)$ is a subgroup of $\operatorname{A}(G,H)$. To prove that $\operatorname{A}^{o_i}_{\operatorname{b}}(G,H)$ is  full, it suffices to show that $\operatorname{A}^{\operatorname{oc}}_+(G,H)$ is full. Let $f,h\in \operatorname{A}^{\operatorname{oc}}_+(G,H)$ and let $g\in \operatorname{A}(G,H)$ be such that $f\leq g\leq h$. For a net $(x_\alpha)_{\alpha\in A}$ with $x_\alpha\downarrow 0$ we have $h(x_\alpha)\downarrow 0$. Since $0\leq f(x_\alpha)\leq g(x_\alpha)\le h(x_\alpha)$ (for every $\alpha\in A$) we conclude  $g(x_\alpha)\downarrow 0$.      
\end{proof}

To show that under the conditions of the Riesz-Kantorovich Theorem \ref{the:RK_final} the sets $\operatorname{A}_{\operatorname{b}}^{\tau_o}(G,H)$ and $\operatorname{A}_{\operatorname{b}}^{o_i}(G,H)$ coincide for $i\in \{1,2,3\}$, we need three technical statements.  

\begin{lemma} \label{lem:Ogasawara_RDP}
	Let $G$ be a partially ordered abelian group that satisfies the Riesz 
	decomposition property. Let $x,y,z \in G$ be such that
	$\{x,y\} \subseteq [0,z]$.
	Then there is $w \in G$ with
	\begin{itemize}
		\item[(i)] $\pm w \leq x$,
		\item[(ii)] $\pm w \leq y$ and
		\item[(iii)] $y-w \leq z - x$.
	\end{itemize}
\end{lemma}

\begin{proof}
	Let $A:=\{-x,-y,x+y-z\}$ and $B:=\{x,y\}$. Since $A\leq B$, the Riesz decomposition property implies the existence of $w \in G$ with $A \leq w \leq B$.	It is straightforward that $w$ satisfies (i), (ii) and (iii).
\end{proof}

\begin{lemma} \label{lem:Ogasawara_existencefancynet}
	Let $G$ be a partially ordered abelian group with the Riesz decomposition property
	and let $H$ be a Dedekind complete lattice-ordered abelian group. Let $f \in \operatorname{A}_{\operatorname{b}}^{\tau_o}(G,H)$
	and $(y_\alpha)_{\alpha \in A}$ be a net in $G$ such that $y_\alpha \downarrow 0$.
	For $\beta \in A$ and $y \in [0,y_\beta]$ there is a net
	$(w_\alpha)_{\alpha \in A_{\geq \beta}}$
	in $G$ such that 
	\begin{itemize}
		\item[(i)] $0 \leq y-w_\alpha \leq y_\beta -y_\alpha$ for every 
		$\alpha \in A_{\geq \beta}$,
		\item[(ii)]  $\inf\{f(w_\alpha); \alpha \in A_{\geq \beta}\}$ exists
		and satisfies 
		$\inf\{f(w_\alpha); \alpha \in A_{\geq \beta}\}\leq 0$.
	\end{itemize}
\end{lemma}

\begin{proof}
	Let $\beta \in A$ and let $y \in [0,y_\beta]$. 
	For $\alpha \in A_{\geq \beta}$ we have 
	$0 \leq y_\alpha \leq y_\beta$. So $\{y_\alpha,y\} \subseteq [0,y_\beta]$.
	By Lemma \ref{lem:Ogasawara_RDP} there is 
	$w_\alpha \in G$ such that 
	$\pm w_\alpha \leq y_\alpha$, $\pm w_\alpha \leq y$ and 
	$y-w_\alpha \leq y_\beta - y_\alpha$ 
	for every $\alpha \in A_{\geq \beta}$. Thus the net
	$(w_\alpha)_{\alpha \in A_{\geq \beta}}$
	satisfies (i).\\
	Next we will show that $\inf\{f(w_\alpha); \alpha \in A_{\geq \beta}\}$ exists.
	Note that $\{w_\alpha; \alpha\in A_{\geq \beta}\} 
	\subseteq [-y,y]$.
	Since $f$ is order bounded, 
	we know that $\{f(w_\alpha); \alpha\in A_{\geq \beta}\}$
	is order bounded in $H$. Thus the Dedekind completeness of $H$ implies the existence of
	$\inf\{f(w_\alpha); \alpha \in A_{\geq \beta}\}$. \\
	It is left to prove that 
	$\inf\{f(w_\alpha); \alpha \in A_{\geq \beta}\}\leq 0$.
	Note that the net $(y_\alpha)_{\alpha \in A}$ satisfies $y_\alpha \downarrow 0$  and that $\pm w_\alpha \leq y_\alpha$ 
	for every $\alpha \in A_{\geq \beta}$. Thus for the net $(w_\alpha)_{\alpha \in A_{\geq \beta}}$ we have $ w_\alpha\xrightarrow{o_1} 0$, and Proposition \ref{pro:basic_convergences} implies
	$w_\alpha \xrightarrow{\tau_o} 0$. Since $f$ is order continuous, it follows that $f(w_\alpha)\xrightarrow{\tau_o}0$.
	Hence Lemma \ref{lem:topologicalconv_inf} implies 
	$\inf\{f(w_\alpha); \alpha \in A_{\geq \beta}\}\leq 0$.
\end{proof}

Due to Theorem \ref{the:RK_final},
the conditions in the subsequent Proposition \ref{pro:Ogasawara_f+ordercontinuous} and Theorem \ref{the:ogasawara_spaces_are_equal} yield \[\operatorname{A}_{\operatorname{b}}(G,H)=\operatorname{A}_{\operatorname{r}}(G,H).\] 
The operator $f^+$ is the positive part of $f$ in the Dedekind complete lattice-ordered abelian group $\operatorname{A}_{\operatorname{b}}(G,H)$, and $f^-$ is the negative part.

\begin{proposition} \label{pro:Ogasawara_f+ordercontinuous}
	Let $G$ be a directed partially ordered abelian group with the Riesz decomposition property and let $H$ be a Dedekind complete lattice-ordered abelian group.
	If $f \in \operatorname{A}_{\operatorname{b}}^{\tau_o}(G,H)$, then
	$f^+, f^- \in \operatorname{A}_+^{\operatorname{oc}}(G,H)$.
\end{proposition}

\begin{proof}
	Let $f \in \operatorname{A}_{\operatorname{b}}^{\tau_o}(G,H)$. We will use Proposition \ref{pro:charAocplus} to show 	$f^+ \in \operatorname{A}_+^{\operatorname{oc}}(G,H)$. Let
	$(y_\alpha)_{\alpha \in A}$ be a net in $X$ such that $y_\alpha \downarrow 0$. From the monotony of $f^+$ it follows that $f^+(y_\alpha)\downarrow $ and that $f^+(y_\alpha)\geq 0$ for every $\alpha \in A$. 
	
	To show that
	$\inf \{f^+(y_\alpha); \alpha \in A\}=0$, let $z$ be a lower bound of $\{f^+(y_\alpha); \alpha \in A\}$.
	Fix $\beta \in A$ and $y \in [0,y_\beta]$. By Lemma
	\ref{lem:Ogasawara_existencefancynet} there is a net
	$(w_\alpha)_{\alpha \in A_{\geq \beta}}$
	in $G$ such that $0 \leq y-w_\alpha \leq y_\beta -y_\alpha$ for every 
	$\alpha \in A_{\geq \beta}$ and such that
	$\inf\{f(w_\alpha); \alpha \in A_{\geq \beta}\}$ exists
	and satisfies 
	$\inf\{f(w_\alpha); \alpha \in A_{\geq \beta}\}\leq 0$.
	For $\alpha \in A_{\geq \beta}$ we can use 
	$0 \leq y-w_\alpha \leq y_\beta -y_\alpha$ to see
	\begin{align*}
	f(y) - f(w_\alpha) = f(y-w_\alpha)\leq f^+(y- w_\alpha)\leq f^+(y_\beta-y_\alpha)= f^+(y_\beta)-f^+(y_\alpha).
	\end{align*}
	Therefore we have shown that
	\begin{align*}
	z \leq f^+ (y_\alpha)
	\leq f^+ (y_\beta) - f(y) + f(w_\alpha)
	\end{align*}
	for every $\alpha \in A_{\geq \beta}$.
	Thus 
	\begin{align*}
	z \leq  f^+ (y_\beta) - f(y) + 
\inf\{f(w_\alpha); \alpha \in A_{\geq \beta}\} \leq  f^+ (y_\beta) - f(y) + 0.
	\end{align*}
	The infimum over $y$ yields
	\begin{align*}
	z &\leq  f^+ (y_\beta) + \inf\{-f(y); y \in [0,y_\beta]\}=  f^+ (y_\beta) - \sup\{f(y); y \in [0,y_\beta]\}\\&= f^+ (y_\beta) - f^+ (y_\beta)=0.
	\end{align*}
	We conclude $\inf \{f^+(y_\alpha); \alpha \in A\}=0$, hence $f^+ \in \operatorname{A}_+^{\operatorname{oc}}(G,H)$. 
	
	Since for $f \in \operatorname{A}_{\operatorname{b}}^{\tau_o}(G,H)$ we have that 
	$-f \in \operatorname{A}_{\operatorname{b}}^{\tau_o}(G,H)$, we obtain 
	$f^-=(-f)^+ \in \operatorname{A}_+^{\operatorname{oc}}(G,H)$.
\end{proof}	

Now we are in a position to present the main results of the present paper in the subsequent two theorems.

\begin{theorem}\label{the:ogasawara_spaces_are_equal}
	Let $G$  be a directed  partially ordered abelian group that satisfies  the Riesz decomposition property and let $H$ be a Dedekind complete lattice-ordered abelian group. Then
	\begin{align*}
\operatorname{A}_{\operatorname{b}}^{o_1}(G,H)&=\operatorname{A}_{\operatorname{b}}^{o_2}(G,H)=\operatorname{A}_{\operatorname{b}}^{o_3}(G,H)=\operatorname{A}_{\operatorname{b}}^{\tau_o}(G,H)\\&=\operatorname{A}_{+}^{\operatorname{oc}}(G,H)-\operatorname{A}_{+}^{\operatorname{oc}}(G,H).\end{align*} 
	 
	\end{theorem}
\begin{proof}
	Let $i\in\{1,2,3\}$.
	By Theorem \ref{thm:ordercontinuous} we have \[\operatorname{A}_{\operatorname{b}}^{o_i}(G,H)\subseteq\operatorname{A}_{\operatorname{b}}^{\tau_o}(G,H).\]
Proposition \ref{pro:Ogasawara_f+ordercontinuous} implies that \[\operatorname{A}_{\operatorname{b}}^{\tau_o}(G,H)\subseteq\operatorname{A}_{+}^{\operatorname{oc}}(G,H)-\operatorname{A}_{+}^{\operatorname{oc}}(G,H).\]
By Proposition \ref{pro:full_subgroup} the set $\operatorname{A}_{\operatorname{b}}^{o_i}(G,H)$ is a subgroup of $\operatorname{A}_{\operatorname{b}}(G,H)$, hence \[\operatorname{A}_{+}^{\operatorname{oc}}(G,H)-\operatorname{A}_{+}^{\operatorname{oc}}(G,H)\subseteq\operatorname{A}_{\operatorname{b}}^{o_i}(G,H).\]
\end{proof}

\begin{theorem}
	\label{the:ogasawara_part2}
	Let $G$  be a directed  partially ordered abelian group that satisfies  the Riesz decomposition property and let $H$ be a Dedekind complete lattice-ordered abelian group. Then 	
$\operatorname{A}_{\operatorname{b}}^{\tau_o}(G,H)$ is an order closed ideal in $\operatorname{A}_{\operatorname{b}}(G,H)$.
\end{theorem}

\begin{proof}
	From Proposition \ref{pro:full_subgroup} and Theorem \ref{the:ogasawara_spaces_are_equal} it follows that $\operatorname{A}_{\operatorname{b}}^{\tau_o}(G,H)$ is a full subgroup of $\operatorname{A}_{\operatorname{b}}(G,H)$.  Proposition \ref{pro:Ogasawara_f+ordercontinuous} implies that $\operatorname{A}_{\operatorname{b}}^{\tau_o}(G,H)$ is closed under the lattice operations in $\operatorname{A}_{\operatorname{b}}(G,H)$. In particular, $\operatorname{A}_{\operatorname{b}}^{\tau_o}(G,H)$ is directed, i.e.\ it is an ideal. 
	
	Combining Theorem \ref{the:ogasawara_spaces_are_equal}, Proposition \ref{pro:Aoc+oclosed} and
	Proposition \ref{pro:McapG+oclosed}, 
	we conclude that $\operatorname{A}_{\operatorname{b}}^{\tau_o}(G,H)$ is order closed.
	\end{proof}

	Theorem \ref{the:ogasawara_part2} is a generalisation of a theorem by Ogasawara \cite{Ogasawara1944} (see also \cite[Theorem 4.4]{Positiveoperators_old}) for $o_1$-continuous operators on vector lattices. 
	
	The following slight generalisation of \cite[Proposition 1.6]{Abra} is obtained due to Theorem \ref{the:ogasawara_spaces_are_equal}.
	
	\begin{corollary}
		\label{cor:orderboundedando2contiso3cont}
		Let $G$ be a directed partially ordered abelian group that satisfies the Riesz decomposition property and let $H$ be an Archimedean lattice-ordered abelian group. Then
		$\operatorname{A}_{\operatorname{b}}^{o_2}(G,H)\subseteq \operatorname{A}_{\operatorname{b}}^{o_3}(G,H)$.
	\end{corollary}
	\begin{proof}
		Let $f \in \operatorname{A}_{\operatorname{b}}^{o_2}(G,H)$ and $(x_\alpha)_{\alpha \in A}$ a net in $G$ such that $x_\alpha \xrightarrow{o_3}x\in G$. Furthermore let $(H^\gamma,J)$ be the group Dedekind completion\footnote{A slight adaptation of arguments given in \cite[Theorem IV.11.1]{Vulikh67} yields that for every Archimedean lattice-ordered abelian group $G$ there is a Dedekind complete lattice-ordered abelian group $G^\gamma$ and an additive order embedding $J\colon G\to G^\gamma$ such that $J[G]$ is order dense in $G^\gamma$. We say that $(G^\gamma,J)$ is the \emph{group Dedekind completion} of $G$.} of $H$.
		Due to Corollary \ref{coro:orderembeddingswithroderdenseimagesarecontinuous} the map $J$ is $o_2$-continuous, hence also $J\circ f\colon G \to H^\gamma$. Since $J\circ f$ is order bounded, Theorem \ref{the:ogasawara_spaces_are_equal} yields $J\circ f\in \operatorname{A}_{\operatorname{b}}^{o_2}(G,H^\gamma)= \operatorname{A}_{\operatorname{b}}^{o_3}(G,H^\gamma)$. Thus $J(f(x_\alpha))\xrightarrow{o_3}J(f(x))$ in $H^\gamma$. Now Proposition \ref{pro:o2_o3_convergenceDedekind_complete_lattice} yields $J(f(x_\alpha))\xrightarrow{o_2}J(f(x))$ in $H^\gamma$. Thus Proposition \ref{pro:Dedekindcompletionando3o2}(ii) shows that $f(x_\alpha) \xrightarrow{o_3}f(x)$.
	\end{proof}

\section{Order convergence and order topology in partially ordered vector spaces} 

In this section let $X$ be a partially ordered vector space. We will show that for $i\in \{1,2,3\}$ the scalar multiplication is jointly continuous with respect to $o_i$-convergence on $X$ and $\mathbb{R}$, respectively, if and only if $X$ is Archimedean and directed. Examples are presented in which the order convergence concepts differ. 

\begin{lemma} \label{lem:characterisationArchimedeananddirectedwithorderconvergence}
	Let $i\in\{1,2,3\}$. Then the following statements are equivalent.
	\begin{itemize}
		\item[(i)] For every $x \in X$ the sequence $(\frac{1}{n}x)_{n \in \N}$ satisfies $\frac{1}{n}x\xrightarrow{o_i}0$.
		\item[(ii)] For every $x \in X$
		the sequence $(\frac{1}{n}x)_{n \in \N}$ satisfies $\frac{1}{n}x\xrightarrow{\tau_o}0$.
		\item[(iii)] $X$ is Archimedean and directed.
	\end{itemize}
	\end{lemma}
	\begin{proof}
		The implication (i)$\Rightarrow$(ii) follows from Proposition \ref{pro:basic_convergences}.
		To show (ii)$\Rightarrow$(iii),
		we first establish $X_+$ to be generating in $X$. Let $x\in X$, then for the sequence $(\frac{1}{n}x)_{n \in \N}$		
		we have $\frac{1}{n}x\xrightarrow{\tau_o}0$, hence Remark \ref{rem:G+-G+orderopen} (c) shows the existence of $n\in\N$ with $\frac{1}{n}x\in X_+-X_+$. Since $X_+-X_+$ is a vector space, we obtain $x\in X_+-X_+$. 
		
		To show that $X$ is Archimedean, let $x\in X_+$. By (ii), we have $\frac{1}{n}x\xrightarrow{\tau_o}0$. Since $(\frac{1}{n}x)\downarrow$, Lemma \ref{lem:topologicalconv_inf} proves $(\frac{1}{n}x)\downarrow 0$.
		
		Next we show (iii)$\Rightarrow$(i).
		Let $x\in X$. By the directedness of $X$, we have  
	$x_1,x_2 \in X_+ $ with $x=x_1-x_2$. Since $X$ is Archimedean, we get 
	$\frac{1}{n}x_j \downarrow 0$ for $j \in \{1,2\}$. 
	Thus $\frac{1}{n}x_j \xrightarrow{o_i} 0$ by Remark \ref{rem:decreasingnet} and Proposition \ref{pro:basic_convergences}. 
	Hence Proposition
	\ref{pro:plus_o_i} implies $\frac{1}{n}x=\frac{1}{n}x_1-\frac{1}{n}x_2\xrightarrow{o_i}0$.
	 \end{proof}

\begin{proposition} \label{pro:characterisation_o_icontinuousscalarmultiplication}
	Let $i \in \{1,2,3\}$. 
	Then the following statements are equivalent.
	\begin{itemize}
		\item[(i)] $X$ is Archimedean and directed.
\item[(ii)] For every net $(\lambda_\alpha)_{\alpha\in A}$ in $\mathbb{R}$ with $\lambda_\alpha\xrightarrow{o_i}\lambda\in \mathbb{R}$ and every net $(x_\beta)_{\beta\in B}$ in $X$ with
$x_\beta\xrightarrow{o_i}x\in X$
 the net  $(\lambda_\alpha x_\beta)_{(\alpha,\beta) \in A\times B}$
 satisfies $\lambda_\alpha x_\beta\xrightarrow{o_i}\lambda x$
  (where $A\times B$ is ordered component-wise).
 
	\item[(iii)] 
	For every net $(\lambda_\alpha)_{\alpha\in A}$ in $\mathbb{R}$ with $\lambda_\alpha\xrightarrow{o_i}\lambda\in \mathbb{R}$ and every net $(x_\alpha)_{\alpha\in A}$ in $X$ with
	$x_\alpha\xrightarrow{o_i}x\in X$
	the net  $(\lambda_\alpha x_\alpha)_{\alpha \in A}$
	satisfies $\lambda_\alpha x_\alpha\xrightarrow{o_i}\lambda x$.
		\end{itemize}
	\end{proposition}
	\begin{proof}
	 To show (i)$\Rightarrow$(ii), let  
	$(\lambda_\alpha)_{\alpha\in A}$ be a net in $\mathbb{R}$ with $\lambda_\alpha\xrightarrow{o_1}\lambda\in \mathbb{R}$ and let  $(x_\beta)_{\beta\in B}$ be a net in $X$ with
	$x_\beta\xrightarrow{o_1}x\in X$.
	According to Proposition \ref{pro:char_o_i_poag}, there is a net $(\check{\lambda}_\alpha)_{\alpha\in A}$ in $\mathbb{R}$ with $\check{\lambda}_\alpha\downarrow 0$ and $\pm (\lambda_\alpha-\lambda)\leq \check{\lambda}_\alpha$ for every $\alpha\in A$, and  a net $(\check{x}_\beta)_{\beta\in B}$ in $X$ with $\check{x}_\beta\downarrow 0$ and $\pm (x_\beta-x)\leq \check{x}_\beta$ for every $\beta\in B$.
	Since $X$ is directed, there is $\check{x}\in X$ with $\pm x\leq \check{x}$. 
	The net 
	$(\check{\lambda}_\alpha \check{x})_{(\alpha,\beta)\in A\times B}$ is a subnet of $(\check{\lambda}_\alpha \check{x})_{\alpha\in A}$, hence  $X$ being Archimedean implies that
	$\check{\lambda}_\alpha \check{x}\downarrow 0$. 
	A straightforward argument shows that the net 
	$(\check{\lambda}_\alpha \check{x}_\beta +\check{\lambda}_\alpha \check{x}+|\lambda|\check{x}_\beta)_{(\alpha,\beta)\in A\times B}$ satisfies $\check{\lambda}_\alpha \check{x}_\beta +\check{\lambda}_\alpha \check{x}+|\lambda|\check{x}_\beta\downarrow 0$.
	 For $(\alpha,\beta)\in A\times B$ we have $\pm \lambda_\alpha\leq \check{\lambda}_\alpha \mp \lambda\leq \check{\lambda}_\alpha+|\lambda|$
	 and hence
	\begin{align*}
	\pm(\lambda_\alpha x_\beta-\lambda x)= \pm \lambda_\alpha (x_\beta- x)\pm (\lambda_\alpha-\lambda)x\leq  (\check{\lambda}_\alpha+|\lambda|)\check{x}_\beta +\check{\lambda}_\alpha \check{x},
	\end{align*}
	such that
	the net  $(\lambda_\alpha x_\beta)_{(\alpha,\beta) \in A\times B}$
	satisfies $\lambda_\alpha x_\beta\xrightarrow{o_1}\lambda x$.
	The arguments for $o_2$-convergence and $o_3$-convergence are similar.
	\\	
	The implication (ii)$\Rightarrow$(iii) follows from Remark \ref{rem:subnet_o_i}. 
	\\
	By Lemma \ref{lem:characterisationArchimedeananddirectedwithorderconvergence} we obtain (iii)$\Rightarrow$(i). 
\end{proof}

Next we present an example of a vector lattice in which  $\tau_o$-convergence and $o_3$-convergence do not coincide.

\begin{example} \label{exa:ordertopconvergentnetnoto3convergent}
	Let $X$ be the vector lattice of all real, Lebesgue-measurable, almost everywhere finite functions on $[0,1]$. As usual, we identify almost everywhere equal functions and order $X$ component-wise almost everywhere. Let $(f_n)_{n \in \N}$ be the sequence of characteristic functions of the intervals \[\textstyle [0,1],[0,\frac{1}{2}],[\frac{1}{2},1],[0,\frac{1}{4}],[\frac{1}{4},\frac{2}{4}],[\frac{2}{4},\frac{3}{4}],[\frac{3}{4},1],[0,\frac{1}{8}],\ldots \] 

 The sequence $(f_n)_{n \in \N}$ does not $o_3$-converge to $0$. Indeed, assume $f_n \xrightarrow{o_3}0$. By Proposition \ref{pro:char_o_i_poag} there is a net $(\check{f}_\alpha)_{\alpha \in A}$ in $X$ with $\check{f}_\alpha \downarrow 0$ and a map $\eta\colon A \rightarrow \N$ such that $\pm f_n \leq \check{f}_\alpha$ for all $\alpha \in A$ and $n \in \N_{\geq \eta(\alpha)}$.	
	To obtain a contradiction  note that 
	$1 =\sup \{f_n; \, n \in \N_{\geq \eta(\alpha)}\}\leq \check{f}_\alpha$ for all $\alpha \in A$. 
	
	We show that $f_n \xrightarrow{\tau_o}0$.
	Let $V\subseteq X$ be order open such that $0\in V$. For $t\in[0,1]$ and $\varepsilon\in \mathbb{R}_{>0}$ let $g^{(t)}_\varepsilon$ be the characteristic function of the interval $[0,1]\cap\left[t-\varepsilon, t+\varepsilon\right]$. Note that for every $t\in[0,1]$ the sequence $\left(g^{(t)}_{\frac{1}{n}}\right)_{n\in\mathbb{N}}$ satisfies $g^{(t)}_{\frac{1}{n}}\downarrow_n 0$. As $V$ is a net catching set for $0$, for every $t\in[0,1]$ there is $\varepsilon(t)\in \mathbb{R}_{>0}$ such that $\left[-g^{(t)}_{\varepsilon(t)},g^{(t)}_{\varepsilon(t)}\right]\subseteq V$. Since $[0,1]$ is compact, there is a finite set $I\subset[0,1]$ such that $\{[t-\varepsilon(t),t+\varepsilon(t)];\, t\in I\}$ is an open cover of $[0,1]$. Let $\delta$ be a Lebesgue number of this cover. There is $n_0\in \mathbb{N}$ such that for every $n \in \mathbb{N}_{\geq n_0}$  the support of $f_n$ has diameter less than $\delta$. Therefore for every $n \in \mathbb{N}_{\geq n_0}$ there is $t \in I$ such that $f_n\in  \left[-g^{(t)}_{\varepsilon(t)},g^{(t)}_{\varepsilon(t)}\right]\subseteq V$. This proves that $f_n \xrightarrow{\tau_o}0$.
\end{example}

As a continuation of Remark \ref{rem:RestrictionandExtensionproperty}, in the subsequent example we present a vector lattice $Y$ with order topology $\tau_o(Y)$ and an order dense subspace $X$ such that the induced topology differs from the order topology $\tau_o(X)$. In the spirit of \cite{IaB} this means, in particular, that the Extension property (E) is not satisfied for order closed sets. 

\begin{example}\label{exa:extensionprop} In \cite[Example 5.2]{IaB} the vector lattice
	\[Y=\left\{y=(y_i)_{i\in \mathbb{Z}} \in l^\infty;\, \lim_{i \rightarrow \infty} y_i \text{ exists}\right\}\]
and its order dense subspace
	\[X=\left\{x=(x_i)_{i \in \mathbb{Z}}\in Y;\, \sum_{k=1}^\infty \frac{x_{-k}}{2^k}=\lim_{i \rightarrow \infty} x_i\right\}\]
	are considered. Moreover, it is shown that the sequence of unit vectors $(e^{(n)})_{n \in \mathbb{N}}$ is $o_1$-convergent to $0$ in $Y$, but is not $o_1$-convergent in $X$. Here for $n,k \in \mathbb{Z}$ we set $e^{(n)}_k:=1$ for $n=k$ and $e^{(n)}_k:=0$ otherwise. 
	
		Let $M:=\{e^{(n)};\, n \in \mathbb{N}\}$. By Theorem \ref{thm:orderclosed}, $M$ is not order closed in $Y$. 
	We will show in (A) that $M$ is order closed in $X$ and in (B) that there is no order closed $N\subseteq Y$ such that $N\cap X=M$. Moreover, in (C) we prove that the sequence $(e^{(n)})_{n \in \mathbb{N}}$ is not convergent with respect to $\tau_o(X)$, and hence not $o_3$-convergent and not $o_2$-convergent.
	
	(A) To show that $M$ is order closed in $X$, we use Theorem \ref{thm:orderclosed}. Let $(n_\alpha)_{\alpha \in A}$ be a net in $\mathbb{N}$ such that $e^{(n_\alpha)}\xrightarrow{o_1}x\in X$. Hence there is a net $(\check{e}^{\alpha})_{\alpha \in A}$ in $X$ such that $\check{e}^\alpha \downarrow 0$ and $\pm\left(e^{(n_\alpha)}-x\right)\leq \check{e}^\alpha$ for all $\alpha \in A$.	We show in the steps (A1) and (A2) that   $(n_\alpha)_{\alpha \in A}$ has exactly one accumulation point $l$, which implies $x=e^{(l)}\in M$. 
	
	(A1) The net  $(n_\alpha)_{\alpha \in A}$ has an accumulation point. 
		
		Indeed, assume the contrary. Let $k \in \mathbb{Z}$. Since no element of $\{0,\ldots,k\}$ is an accumulation point of $(n_\alpha)_{\alpha \in A}$, there is $\alpha_k \in A$ such that for every $\alpha \in A_{\geq \alpha_k}$ we have $n_\alpha > k$. Hence $e^{(n_\alpha)}_k=0$ for every $\alpha \in A_{\geq \alpha_k}$, and $|x_k|=\left|e^{(n_\alpha)}_k-x_k\right|\leq  \check{e}_k^\alpha \downarrow 0$ implies $x_k=0$. This shows $x=0$. 
	
	We show that $\lim_{k \rightarrow \infty}\check{e}^\alpha_k \geq 1$ for every $\alpha \in A$. Assuming $\lim_{k \rightarrow \infty}\check{e}^\alpha_k < 1$, for every $\alpha \in A$ there is $K \in \mathbb{N}$ such that for every $k \in \mathbb{N}_{\geq K}$ we have $\check{e}_k^\alpha<1$. Since $(n_\beta)_{\beta \in A}$ has no accumulation points,  there is $\beta\in A_{\geq \alpha}$ such that $n_\beta \geq K$, and we obtain the contradiction $1>\check{e}^\alpha_{n_\beta}\geq \check{e}^{\beta}_{n_\beta}\geq e^{(n_\beta)}_{n_\beta}=1$.
	
	We do not have $\check{e}_k^\alpha \downarrow_\alpha 0$ for every $k \in \mathbb{Z}\setminus \mathbb{N}$, since otherwise monotone convergence would imply $1\leq \lim_{k \rightarrow \infty}\check{e}^\alpha_k= \sum_{k=1}^\infty \frac{\check{e}_{-k}^\alpha}{2^k} \downarrow_\alpha 0$. Hence there is $k \in \mathbb{Z}\setminus \mathbb{N}$ and $\delta>0$ with $\check{e}_k^\alpha\geq \delta$ for every $\alpha \in A$. Put $w:=\delta e^{(k)}-2 \delta e^{(k-1)}
	$ and observe the contradiction $w\leq \check{e}^\alpha \downarrow_\alpha 0$. This shows that $(n_\alpha)_{\alpha \in A}$ has accumulation points. 
	
	(A2) The net  $(n_\alpha)_{\alpha \in A}$ has at most one  accumulation point.
	
	Indeed, let $l,k\in \mathbb{N}$ be accumulation points of this net. As $\check{e}_l^{\alpha}\downarrow 0$, we obtain that for every $\epsilon>0$ there is an $\alpha_0 \in A$ such that for every $\alpha \in A_{\geq \alpha_0}$ we have $\left|e^{(n_\alpha)}_l-x_l\right|\leq \check{e}_l^\alpha \leq \check{e}_l^{\alpha_0} \leq \epsilon$. Since $l$ is an accumulation point of $(n_\alpha)_{\alpha \in A}$, there is $\alpha \in A_{\geq \alpha_0}$ such that $l=n_\alpha$. Thus $|1-x_l|=\left|e^{(l)}_l-x_l\right|=\left|e^{(n_\alpha)}_l-x_l\right|\leq \epsilon$, consequently $x_l=1$. Since $k$ is an accumulation point of $(n_\alpha)_{\alpha \in A}$, there is $\beta \in A_{\geq \alpha_0}$ such that $k=n_\beta$. Hence $\left|e^{(k)}_l-1\right|=\left|e^{(n_\beta)}_l-x_l\right|\leq \epsilon$ and we have shown $e^{(k)}_l=1$, i.e.\ $k=l$.
	
	(B) To show that there is no order closed set  $N\subseteq Y$ such that $N\cap X=M$, assume the contrary. As $e^{(n)}\xrightarrow{o_1}0$ we obtain $0\in N$. Hence $0\in N\cap X=M$, which is a contradiction.  
	
	(C) Assume that $e^{(n)}\xrightarrow{\tau_o(X)}x \in X$. Since $M$ is order closed in $X$, there is $l \in \mathbb{N}$ such that $x =e^{(l)}$. Let $O:=\{x \in X;\, x_l\in (0,2)\}$ and observe that $e^{(l)}\in O\in \tau_o(X)$. Thus $e^{(n)}\xrightarrow{\tau_o(X)}e^{(l)}$ implies the existence of $N\in \mathbb{N}$ such that for every $n \in \mathbb{N}_{\geq N}$ we have $e^{(n)}\in O$, a contradiction. 
\end{example}

\section{Properties of the set of order continuous linear operators in partially ordered vector spaces}

In this section, let $X$ and $Y$ be  partially ordered vector spaces. In this setting, we provide similar statements as in Section 7. 
The following is a slight generalisation of \cite[Theorem 2.1]{Abra}. Note that for $i=1$ the result is contained in Proposition \ref{pro:o1ob}.

\begin{proposition}
	\label{pro:o_icontinuousisorderbdd}
	Let $X$ be Archimedean, $G$ be a partially ordered abelian group and $i \in \{1,2,3\}$. Every $o_i$-continuous and additive map $f\colon X \rightarrow G$ is order bounded. 
\end{proposition} 

\begin{proof}
	Note that it is sufficient to show that $f[[0,v]]$ is order bounded in $G$ for every $v \in X_+$. Let $A:=\mathbb{N}\times [0,v]$ be ordered lexicographically and define $x_{(n,w)}:=\frac{1}{n}w$, $\hat{x}_{(n,w)}:=-\frac{1}{n}v$ and $\check{x}_{(n,w)}:=\frac{1}{n}v$ for $(n,w)\in A$. Note that $\hat{x}_\alpha \uparrow 0$ and $\check{x}_\alpha \downarrow 0$ and that $\hat{x}_\alpha \leq x_\alpha \leq \check{x}_\alpha$ for all $\alpha \in A$. Thus $x_\alpha \xrightarrow{o_1}0$. Since $f$ is $o_i$-continuous, by Proposition \ref{pro:basic_convergences} we obtain $f(x_\alpha)\xrightarrow{o_3}0$. Therefore by Proposition \ref{pro:char_o_i_poag}(iii) there is a net $(y_\beta)_{\beta \in B}$ and a map $\eta \colon B \rightarrow A$ such that $y_\beta \downarrow 0$ and $\pm f(x_\alpha)\leq y_\beta$ for every $\beta \in B$ and $\alpha \in A_{\geq \eta(\beta)}$. Fix $\beta\in B$. Since $\eta(\beta)\in A$ there are $(m,u)\in A$ such that $\eta(\beta)=(m,u)$. Now let $w \in [0,v]$ and observe that $(m+1,w)\geq (m,u)=\eta(\beta)$. Thus 
	$\pm f(w)=\pm(m+1)f\left(\frac{1}{m+1}w\right)=\pm(m+1)f\left(x_{(m+1,w)}\right)\leq (m+1) y_\beta$. Hence $f[[0,v]]\subseteq [-(m+1)y_\beta,(m+1)y_\beta]$. 
\end{proof}

\begin{remark}
	It is an open question whether Proposition \ref{pro:o_icontinuousisorderbdd} is valid if $X$ is an Archimedean partially ordered abelian group.
\end{remark}

We denote
$\operatorname{L}^{o_i}_{\operatorname{b}}(X,Y)=
\operatorname{A}^{o_i}_{\operatorname{b}}(X,Y)\cap \operatorname{L}(X,Y)$, $\operatorname{L}^{\tau_o}_{\operatorname{b}}(X,Y)=
\operatorname{A}^{\tau_o}_{\operatorname{b}}(X,Y)\cap \operatorname{L}(X,Y)$
and
$\operatorname{L}_+^{\operatorname{oc}}(X,Y)=
\operatorname{A}_+^{\operatorname{oc}}(X,Y)\cap \operatorname{L}(X,Y)$.
The proof of the following statement is similar to the one in \cite[Lemma 1.26]{CAD}.
\begin{proposition}
	\label{pro:additive_mon_implies_homogen}
	If $X$ is directed and $Y$ is Archimedean,
	then every additive monotone map is homogeneous, i.e.\ $\operatorname{A}_+(X,Y)=\operatorname{L}_+(X,Y)$.
\end{proposition}

An analogue for $o_i$-continuous maps is given next.

\begin{proposition}
	\label{pro:additive_o_i_cont_implies_homogen}
	Let $X, Y$ be directed and Archimedean and let $i\in\{1,2,3\}$. Then every additive $o_i$-continuous map from $X$ to $Y$ is homogeneous, hence
	$\operatorname{A}^{o_i}_{\operatorname{b}}(X,Y)=\operatorname{L}^{o_i}_{\operatorname{b}}(X,Y)$. Furthermore, $\operatorname{A}_+^{\operatorname{oc}}(X,Y)=\operatorname{L}_+^{\operatorname{oc}}(X,Y)$.	 
\end{proposition}

\begin{proof} Let $T\in \operatorname{A}^{o_i}_{\operatorname{b}}(X,Y)$. 
	Observe that every additive maps is $\mathbb{Q}$-homo\-geneous.
	Let $\lambda\in\mathbb{R}$ and $x\in X$.
	There is a sequence $(\lambda_n)_{n\in\mathbb{N}}$ in $\mathbb{Q}$ that $o_i$-convergences to $\lambda$ (with respect to $\mathbb{R}$, cf.\ Example  \ref{exa:opensubsetsofR}). By Proposition \ref{pro:characterisation_o_icontinuousscalarmultiplication} we get $\lambda_n x\xrightarrow{o_i} \lambda x$ and $\lambda_n T(x)\xrightarrow{o_i} \lambda T(x)$.
	Since $T$ is $o_i$-continuous, we obtain $T(\lambda_n x)\xrightarrow{o_i} T(\lambda x)$.
		As $T$ is $\mathbb{Q}$-homogeneous, we get for every $n\in\mathbb{N}$ that $T(\lambda_n x)=\lambda_n T(x)$. 
	Due to Remark 
	\ref{rem:unique_order_limits}
	order limits are unique, hence we conclude $T(\lambda x)=\lambda T(x)$. 
\end{proof}

  Under the conditions of Proposition \ref{pro:additive_mon_implies_homogen}, 
  we obtain 
   \[\operatorname{A}_+(X,Y)-\operatorname{A}_+(X,Y)=
   \operatorname{L}_+(X,Y)-\operatorname{L}_+(X,Y)\subseteq
   \operatorname{L}_{\operatorname{b}}(X,Y)\subseteq \operatorname{A}_{\operatorname{b}}(X,Y).\] Hence, if $\operatorname{A}_{\operatorname{b}}(X,Y)$ is directed, then $\operatorname{A}_{\operatorname{b}}(X,Y)=\operatorname{L}_{\operatorname{b}}(X,Y)$. Therefore, Theorem \ref{the:RK_final} yields the following statement.  
  \begin{theorem}
	Let $X$  be a directed   partially ordered vector space with  the Riesz decomposition property, and let $Y$ be a Dedekind complete vector lattice. Then every additive 
	order bounded map is homogeneous, i.e.\ $\operatorname{A}_{\operatorname{b}}(X,Y)=\operatorname{L}_{\operatorname{b}}(X,Y)$. 
\end{theorem}


We reformulate the Theorems \ref{the:ogasawara_spaces_are_equal} and \ref{the:ogasawara_part2} and obtain a generalisation of the Ogasawara theorem.  

\begin{theorem}\label{the:finalOga}
	Let $X$  be a directed   partially ordered vector space with the Riesz decomposition property, and let $Y$ be a Dedekind complete vector lattice.
Then 
\begin{align*}
\operatorname{L}_{\operatorname{b}}^{o_1}(X,Y)&=\operatorname{L}_{\operatorname{b}}^{o_2}(X,Y)=\operatorname{L}_{\operatorname{b}}^{o_3}(X,Y)=\operatorname{L}_{\operatorname{b}}^{\tau_o}(X,Y)\\&=\operatorname{L}_{+}^{\operatorname{oc}}(X,Y)-\operatorname{L}_{+}^{\operatorname{oc}}(X,Y).\end{align*} 
	 Moreover, 	
	$\operatorname{L}_{\operatorname{b}}^{\tau_o}(X,Y)$ is an order closed ideal in $\operatorname{L}_{\operatorname{b}}(X,Y)$.
\end{theorem}
If $X$ is, in addition, Archimedean, then by Proposition \ref{pro:o_icontinuousisorderbdd} and Theorem \ref{the:finalOga} a linear operator $T\colon X \to Y$ is $o_i$-continuous if and only if $T\in \operatorname{L}_{+}^{\operatorname{oc}}(X,Y)-\operatorname{L}_{+}^{\operatorname{oc}}(X,Y)$. 

It is an open question whether one obtains similar results to the ones in Theorem \ref{the:finalOga} under weaker assumtions. 
In particular, if $Y$ is an Archimedean vector lattice, but not Dedekind complete, then the set of all regular linear operators is an Archimedean directed partially ordered vector space, and the notion of an ideal is at hand, see \cite{IaB}. One can ask whether the set of order continuous (or $o_i$-continuous)  regular linear operators is an order closed ideal in the space of regular operators.

 \vspace{10mm} \noindent
\begin{tabular}{l l l }
Till Hauser &Anke Kalauch  \\
Fakult\"at f\"ur Mathematik und Informatik& FR Mathematik \\
Institut f\"ur Mathematik&Institut f\"ur Analysis\\
Friedrich-Schiller-Universit\"at Jena&TU Dresden \\
Ernst-Abbe-Platz 2&\\
07743 Jena&01062 Dresden\\
Germany&Germany&\\
till.hauser@uni-jena.de & \end{tabular}

\end{document}